\documentclass[11pt,twoside,a4paper]{article}
\usepackage{tikz}
\usepackage{charter}
\usepackage{amsthm}
\usepackage{amsmath,stmaryrd, mathrsfs}
\usepackage{amsfonts}
\usepackage{amssymb}
\usepackage{dsfont}
\usepackage{tikz-cd}
\usepackage{subcaption}
\usepackage{float}
\usepackage{url}
\usepackage{hyperref}
\usepackage{empheq}
\usepackage{booktabs}
\usepackage{import}
\usepackage{enumitem}
\usepackage[normalem]{ulem}
\usepackage{soul}
\usepackage{mhequ}

\usepackage{upgreek}

\usepackage{wrapfig}
\usepackage{pgfplots}

\usepackage{mathabx}

\usepackage[margin=2.0cm,top=3cm, bottom=2.7cm]{geometry}

\newcommand{\bigsquare}{\mathord{\vcenter{\hbox{\scalebox{1.8}{$\square$}}}}}

\usepackage{fancyhdr}

\newtheorem{theorem}{Theorem}[section]
\newtheorem{lemma}[theorem]{Lemma}
\newtheorem{proposition}[theorem]{Proposition}
\newtheorem{corollary}[theorem]{Corollary}
\newtheorem{remark}[theorem]{Remark}

\newtheorem{definition}[theorem]{Definition}

\newtheorem{example}{Example}

\newcommand*{\ud}{\mathrm{\,d}}

\newcommand{\RR}{\BF{R}}
\newcommand{\NN}{\BF{N}}
\newcommand{\ZZ}{\BF{Z}}

\newcommand{\TT}{\BF{T}}

\newcommand{\EE}{\BF{E}}

\newcommand{\PP}{\BF{P}}

\newcommand{\mN}{\mathcal{N}}
\newcommand{\mO}{\CB{O}}

\newcommand{\mC}{\mathcal{C}}

\newcommand{\mM}{\CB{M}}

\newcommand{\mD}{\mathcal{D}}

\newcommand{\mR}{\CB{R}}

\newcommand{\tOne}{\raisebox{-0.5ex}{\includegraphics[scale=1.2]{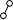}}}

\DeclareMathAlphabet{\BF}{OT1}{ptm}{bx}{n}
\DeclareMathAlphabet{\CB}{U}{BOONDOX-calo}{m}{n}

\newcommand{\rms}{\mathrm s}
\newcommand{\eqdef}{\stackrel{\mbox{\tiny\rm def}}{=}}
\newcommand{\rmd}{\mathrm d}
\newcommand{\Norm}[1]{{\talloblong #1 \talloblong}}

\let\emph\relax
\DeclareTextFontCommand{\emph}{\bfseries}

\newcommand{\nnorm}[1]{{\vert\kern-0.25ex\vert\kern-0.25ex\vert #1 
    \vert\kern-0.25ex\vert\kern-0.25ex\vert}}

\makeatletter 
\newcommand*{\bigcdot}{}
\DeclareRobustCommand*{\bigcdot}{%
  \mathbin{\mathpalette\bigcdot@{}}%
}
\newcommand*{\bigcdot@scalefactor}{.5}
\newcommand*{\bigcdot@widthfactor}{1.15}
\newcommand*{\bigcdot@}[2]{%
  \sbox0{$#1\vcenter{}$}
  \sbox2{$#1\cdot\m@th$}%
  \hbox to \bigcdot@widthfactor\wd2{%
    \hfil
    \raise\ht0\hbox{%
      \scalebox{\bigcdot@scalefactor}{%
        \lower\ht0\hbox{$#1\bullet\m@th$}%
      }%
    }%
    \hfil
  }%
}
\makeatother

\title{Discrete RG flow of a hierarchical singular SPDE}
\author{L\'eonard Ferdinand\footnote{Email: \href{mailto:leonard.ferdinand@mis.mpg.de}{leonard.ferdinand@mis.mpg.de}, Max Planck Institute for Mathematics in the Sciences, Leipzig, Germany.} 
\ and\  Simon Gabriel\footnote{Email: \href{mailto:s.gabriel@berkeley.edu}{s.gabriel@berkeley.edu},
Department of Mathematics, University of California, Berkeley, USA.}
}
\date{ }

\begin{document}
\pagestyle{fancy}

\maketitle

\vspace{-1cm}

\begin{abstract}
    This paper illustrates the Renormalisation Group (RG) approach to singular SPDEs following a framework introduced by Kupiainen \cite{Kupiainen2016}. 
    We study a linear elliptic SPDE with a hierarchical Laplace operator and multiplicative noise, in two dimensions.
    Although this model is a significant simplification, it captures the core mechanisms of the RG method in a transparent setting.
    Particular emphasis is placed on the dynamical system governing the flow of the effective force coefficients, the central object of the RG method.
\end{abstract}

\vspace{-.1cm}

\noindent\hspace{1cm}{\small\textit{{Keywords.}} Renormalisation Group; Singular SPDEs; Anderson model.}

\noindent\hspace{1cm}{\small\textit{{MSC classification.}} Primary: 60H15, Secondary: 35R60; 81T17.}

\vspace{-.35cm}

\setcounter{tocdepth}{2}
\begingroup
  \hypersetup{hidelinks}
  \tableofcontents
\endgroup

\section{Introduction}

In \cite{Kupiainen2016}, Kupiainen introduced a Renormalisation Group (RG) approach to
singular SPDEs, providing 
an alternative to the theories of Regularity Structures \cite{Hai14} and Paracontrolled Distributions \cite{GIP15}.
More recently, interest in RG approaches to singular SPDEs has grown further, 
in particular after Duch introduced a continuous alternative \cite{Duch21}.
Kupiainen’s original treatment focussed on the dynamic $\varphi^4_3$ equation \cite{Kupiainen2016},
and the method was subsequently extended to an anisotropic KPZ equation in \cite{KM17}.
Similar to Regularity Structures and Paracontrolled Distributions, the RG approach relies on
the introduction of a new reference scale, which is larger than the regularisation scale,
and facilitates the construction of the solution.
In the context of RG, the strategy is to derive the effective equation solved by the
restriction of the solution to this reference scale. 
Effective equations at different scales are then related through a dynamical system. 

Kupiainen's work requires a familiarity with Cauchy estimates and rescalings
which are common in the RG community, but less so in the SPDE community. 
Much of the technical difficulty stems from the presence of nonlocal 
terms in the evolution of the effective equation, which have to be disentangled from their
local counterparts via localisation procedures.
In the present work, we aim to clarify the strategy of the RG framework
by setting aside the difficulties associated with any nonlocal behaviour. 
To this end, we work with a hierarchical model, for which the effective 
equation remains local at every scale.
More specifically, we study the hierarchical Anderson equation, which for $ g \in \RR$ is given by
\begin{equation}\label{e_anderson}
    \begin{aligned}
        0=
        \Delta_{H} u (x)  +
        (g \xi(x) - \infty ) u (x)   + 1\,, \qquad x \in \TT^{2}\,,
    \end{aligned}
\end{equation}
where $ \Delta_{H}$ denotes the hierarchical Laplace operator, which we
introduce in the next section. 
Moreover, $ \TT^{2} \eqdef [-1/2,1/2)^{2}$ denotes the two--dimensional torus with periodic
boundary conditions, and $ \xi$ is a white noise on $\TT^{2}$ on a fixed probability space $ (\Omega, \CB F , \PP) $.

The hierarchic component of our model allows
for major simplification, avoiding any localisation procedure, 
or Cauchy estimates.
While the SPDE \eqref{e_anderson} is linear, it allows us to keep notation to a
minimum, and still provide the
key conceptual ideas of \cite{Kupiainen2016}.
We also unfold with great care how to obtain the dynamical
system~\eqref{e_rg_system} that governs the evolution of the effective equation, when
considered on a unit reference scale. 
In Kupiainen’s work, this structure is treated less explicitly.
The formulation of the dynamical system~\eqref{e_rg_system} will be familiar 
to readers with a background in Renormalisation Group theory. 
Indeed, we show that the discrete flow of an effective equation across scales 
is analogous to the discrete flow of an effective action in the context of the RG,
see for example, \cite[Section~5.3.1]{rgBook}. 
We therefore emphasise clarifying the relation between these two perspectives 
and facilitating communication across communities.

We now turn to the analytical difficulties posed by the equation itself.
The SPDE \eqref{e_anderson} contains the singular product $ \xi u $, which
cannot be made sense of by classical means, and needs to given meaning through
approximation. 
To this end, we introduce the discrete tori
\begin{equation*}
    \begin{aligned}
        {\Lambda}_{-N}\eqdef\big\{-\tfrac{1}{2}(1-L^{-N}),
        \ldots,-2L^{-N},-L^{-N},0,L^{-N},2L^{-N},\dots, \tfrac{1}{2}(1-L^{-N}) \big\}^{d}
        \subset \big[-\tfrac{1}{2},\tfrac{1}{2}\big)^d
        \,.
    \end{aligned}
\end{equation*}
The regularised version of \eqref{e_anderson} on the lattice $\Lambda_{-N}$,
referred to as the \emph{bare equation}, then takes the form
\begin{equation}\label{e_anderson_mol}
    (- \underline{\Delta}_{H} )u_{ N} = ( g\,\underline \xi{}_{N} + r_{N}) u_{ N} +1\,,
\end{equation}
with a postulated \emph{bare mass} of the form $r_N\equiv r_N(r,g)\eqdef r-
(1-L^{-2})Ng^{2}$, for some $ r \geqslant 0 $. 
Here, $u_N\in  \mC (\Lambda_{-N})$ and $  \underline{\Delta}_{H} \equiv \underline{\Delta}_{H,-N}$  denotes the hierarchical Laplace operator on $
    {\Lambda}_{-N}$.
The regularised noise
$ \underline\xi{}_{N}$ is given in terms of i.i.d. Gaussian random variables
\begin{equation*}
    \begin{aligned}
        \underline{\xi}{}_{N} ( x) \sim \mN (
        0 , L^{2(2-\alpha)N})\,, \quad x \in
        {\Lambda}_{-N}\,,
    \end{aligned}
\end{equation*}
where $ \alpha \eqdef 2 -
    d/2$ denotes the expected regularity of the solution $u$.
Furthermore, we embed $ u_N$ into $ \TT^d$ by linear interpolation between lattice points.

Our main result is as follows.
\begin{theorem}\label{thm_main}
    Let $d = 2$ and $r \geqslant 1$.
    There exists an odd $L \in \NN$, sufficiently large and depending only on~$r$,
    such that the following holds.

    For every $g \in (0,1]$, there exists an event $\Omega_g \subset \Omega$ such
    that $\PP(\Omega_g) \geqslant 1 - C e^{-c g^{-2}}$.
    On $\Omega_g$, the SPDE~\eqref{e_anderson_mol} with renormalised mass $r$ and
    renormalised coupling constant $g$  admits a unique solution
    $u_N$ on $\Lambda^N$, for every $N \in \NN$.
    Moreover, there exists a function $u^\star \in \CB C^{1-}(\TT^{2})$ such that
    \begin{equation*}
        \begin{aligned}
            u_N \to u^\star
            \qquad \text{in } \CB C^{1-}(\TT^{2}),
            \quad \text{as } N \to \infty \,.
        \end{aligned}
    \end{equation*}
    The H\"older norm $\| \bigcdot \|_{\CB C^{1-\kappa}}$ is defined in
    Definition~\ref{def_besov}, and the events
    $\Omega_g$ are characterised in \eqref{eq_def_Omega}.
\end{theorem}

The bare mass $ r_{N}$ is the sum of the \emph{renormalised mass} $r$ and the \emph{mass counterterm}
$-(1-L^{-2})Ng^2$.
The latter diverges as $N\uparrow\infty$, in agreement with the presence of the term $-\infty u$ in \eqref{e_anderson}.
Moreover, we notice that while $u$ in \eqref{e_anderson} appeared to be fully characterised by
the \emph{renormalised coupling constant} $g$, it is in fact necessary to specify
two quantities to identify $u_N$, and its limit $ u^{\star}$: The renormalised
coupling constant $g\in\RR$ and the renormalised mass $r\in\RR$.

Since we only consider the massive equation, we assume that $r \geqslant 1$.
Moreover, the condition on a small coupling $g\in(0,1]$ is necessary to deal with the
large field problem.
Of course, taking $g\in[-1,0)$ or $r\leqslant-1$ would not bring any modification,
but we choose to fix the signs to simplify the notation.\\

The proof of Theorem~\ref{thm_main} will be given at the end of Section~\ref{sec_unit_lattice_v}.
The strategy is to derive the effective equation solved by a coarse grained version of $u_N$ at a
scale $L^{-n}\geqslant L^{-N}$. 
The key step is the derivation of a dynamical system which relates
the effective equation at scale $L^{-n}$ to the effective equation at the subsequent scale
$L^{-n+1}$. 
At each step of the RG, we rescale the effective equation back to the
unit lattice, which is necessary to compare across different scales in a
meaningful way, see also Remark~\ref{rem_whyunit2}.

\begin{remark}[Random coupling]
	Theorem~\ref{thm_main} is equivalent to saying that there exists a random
	coupling $g(\omega)$, $ \omega \in \Omega$, with negative $p$-th moment growing
	like $\sqrt{p}$, such that the equation \eqref{e_anderson_mol} is almost surely well-posed. 
	In fact, the true quantity that needs to be taken random, in order to avoid the
	large field problem, is the quotient $g/r$. In other words, instead of a small random coupling 
	$g(\omega)$, it would equally be possible to take a large random mass $r(\omega)$. 
	We refrain from doing so to keep the notation lighter, and because the former choice is more common~\cite{Duch22}. 
\end{remark}

   The SPDE \eqref{e_anderson} is the elliptic (and hierarchical) analogue of the parabolic
Anderson model
\begin{equation*}
    \begin{aligned}
        ( \partial_{t} - \Delta ) u = (g\xi - \infty) u \,,
    \end{aligned}
\end{equation*}
describing the evolution of a particle in a random medium, see for example
\cite{Hai14,GIP15,hairer2015simple} among
others. In the Euclidean setting, the elliptic equation \eqref{e_anderson} has been 
studied previously for $
d=2,3$, see for example \cite{allez2015continuous,Lab19,broux2023example}.
In $ d=4$, the equation is scaling critical, and falls outside the reach of subcritical
theories. Fluctuations of the critical (Euclidean) equation are investigated in~\cite{GabRos26}.

In the present paper, we restrict attention to the case $ d = 2 $
and do not treat nonlinearities of the form $\sigma(u)\xi$.
However, the method should extend to the full subcritical regime ($ \alpha>0$) under strong enough assumptions on the nonlinearity 
$\sigma$, at the expense of enlarging the noise structure and 
incorporating additional stochastic terms into the dynamical system.
Lastly, our assumption that the noise $\xi$ is Gaussian is not essential, however, convenient as we use hypercontractivity for the necessary moment estimates in Lemma~\ref{lem_stoch}.

\subsection*{Outline of the paper}

We begin in Section~\ref{sec_setup} by introducing the lattice operations, in particular the hierarchical Laplacian.
In Section~\ref{sec_unit_lattice_v}, we set the stage by
rescaling the problem to the unit scale, which allows for comparison across scales, and
prove Theorem~\ref{thm_main}.
We place particular emphasis on the derivation of the force coefficients and their connection to the effective equation.
Next, we perturbatively derive the dynamical system of force coefficients in
Section~\ref{sec_pert}, which is the key idea of RG. 
In Section~\ref{sec_rig_RG} we will make this perturbative analysis rigorous. 
We conclude the article by proving convergence of the effective force and the
renormalised solutions in 
Section~\ref{sec_conv}.

\subsection*{Acknowledgements}

The authors express their gratitude  towards Ajay Chandra and Hendrik Weber for
suggesting the investigation of  hierarchical models, as well as helpful
discussions.

This material is based upon work supported by the National Science Foundation under Grant No. DMS-1928930, while the authors were in residence at the Simons Laufer Mathematical Sciences Institute in Berkeley, California, during the fall semester of 2025.
Moreover, SG acknowledges financial support through the ERC Grant 101045082, held by Hendrik Weber, and the DFG
Excellence Cluster in Münster EXC 2044–390685587.

\section{Background on hierarchical lattice operators}\label{sec_setup}

In this section, we introduce the regularised lattice and the hierarchical Laplace
operator, closely following \cite{rgBook}.
We then set the framework by defining rescaling operations and a hierarchical analogue of Besov spaces.
As a first application, we regularise the white noise
$ \xi$ and show that it belongs to the natural Besov space, which also allows
the reader to become familiar with the concepts.

\begin{definition}[Lattices]
    We fix a coarse graining scale $L \geqslant 3 $, which is an odd\footnote{
        We could also consider even integers, however, it is convenient to consider $L$
odd such that each square $\bigsquare$ has a centre point. Since $L$ is chosen large
eventually, we did not strive for the greatest generality here.}
    integer, eventually taken sufficiently large.
    Then for any pair of integers $m \leqslant M$, we set
    \begin{equs}
        \Lambda_{m}^M\eqdef  \Big(\frac{L^m\ZZ}{L^M\ZZ}\Big)^d \subset \big[-
        \tfrac{1}{2} L^M , \tfrac{1}{2} L^M\big)^d\,,
    \end{equs}
    to be the lattice of side-length $L^{M}$ and spacing $ L^{m}$, which we identify with
    \begin{equs}
        \Lambda_m^M  \equiv
        \big\{-\tfrac{1}{2}(L^M-L^m), \ldots,-2L^m,-L^m,0,L^m,2L^m,\dots,
        \tfrac{1}{2}(L^{M}-L^m) \big\}^{d}  \,.
    \end{equs}
    Importantly, we have $\Lambda_m\subset \TT^d$, and $\Lambda^{M}\subset L^M\TT^d$.
    Whenever one of these indices equals zero, we drop it from the notation, and
    write
    $\Lambda_m\equiv\Lambda_m^0$, for  $m\leqslant0$, and
    $\Lambda^M\equiv\Lambda_0^M$, for $M\geqslant0$.
\end{definition}

We equip $ \Lambda_{m}^{M}$ with the product topology, therefore,
we denote the set of bounded functions $ \Lambda_{m}^{M} \to \RR$, which are
trivially continuous, by $ \mC (\Lambda_{m}^{M})$.
For $n\in\NN$, and for every $x\in\Lambda_{-n}$, we denote by
$\bigsquare_{-n}(x)\subset\TT^d$ the square of side-length $L^{-n}$
centred at $x$.
For example, $ \Lambda_{-2}$ with $L=3$ can be represented in terms of 

\begin{equation*}
	\begin{aligned}
\begin{tikzpicture}[x=1cm,y=1cm, line cap=round, line join=round, scale=0.75]

\def\Outer{6.75}      
\def\M{0}             
\def\Fi{0.35}         
\def\s{1.6}           
\def\Sp{0.6}          
\def\dotr{2.2pt}      
\def\di{0.20}         

\draw (\M,\M) rectangle ++(\Outer,\Outer);

\foreach \r in {0,1,2}{
  \foreach \c in {0,1,2}{

    \pgfmathsetmacro{\x}{\M + \Fi + \c*(\s+\Sp)}
    \pgfmathsetmacro{\y}{\M + \Fi + \r*(\s+\Sp)}

    \draw (\x,\y) rectangle ++(\s,\s);

    \pgfmathsetmacro{\step}{(\s-2*\di)/2}

    \foreach \dr in {0,1,2}{
      \foreach \dc in {0,1,2}{
        \pgfmathsetmacro{\dx}{\x + \di + \dc*\step}
        \pgfmathsetmacro{\dy}{\y + \di + \dr*\step}
        \fill (\dx,\dy) circle (\dotr);
      }
    }
  }
}


\node[anchor=east] (lbl0) at (\M-1, \M+0.5*\Outer) {$\bigsquare_0$};
\draw[->, thick] (lbl0.east) -- (\M-0.1, \M+0.5*\Outer);

\pgfmathsetmacro{\xMR}{\M + \Fi + 2*(\s+\Sp)}
\pgfmathsetmacro{\yMR}{\M + \Fi + 1*(\s+\Sp)}
\coordinate (t1) at (\xMR+\s, \yMR+0.5*\s-.4);

\node[anchor=west] (lbl1) at (\M+\Outer+0.6, \yMR+0.5*\s-.7) {$ \bigsquare_{-1}(x)$};
\draw[->, thick] (lbl1.west) -- (t1);

\pgfmathsetmacro{\xTR}{\M + \Fi + 2*(\s+\Sp)}
\pgfmathsetmacro{\yTR}{\M + \Fi + 1*(\s+\Sp)}
\pgfmathsetmacro{\step}{(\s-2*\di)/2}
\coordinate (t2) at (\xTR+\di+\step+.1, \yTR+\di+\step+.05); 

\node[anchor=west] (lbl2) at (\M+\Outer+0.6, \yTR+\di+\step+0.8) {$ x = \bigsquare_{-2}(x)$};
\draw[->, thick] (lbl2.west) -- (t2);

\end{tikzpicture}
\end{aligned}
\end{equation*}

\subsection{Averaging and fluctuation operators}

In this section, we introduce the operators needed to define the hierarchical Laplace operator. 
We begin with the averaging operator acting on distributions on $\TT^{d}$, which can be viewed 
as testing a distribution against a constant test function.

\begin{definition}[Averaging operators]
    For $n \in \NN$, we define the \emph{averaging operator}
    $\underline Q{}_{-n}: \mD'(\TT^d)\rightarrow  \mC
        (\Lambda_{-n})$ with
    \begin{equs}
        \underline Q{}_{-n}f(x) \eqdef L^{dn } \int_{\bigsquare_{-n}(x)} f(y)\rmd y \,,\quad x\in\Lambda_{-n}\,.
    \end{equs}
    In particular, we have $ \underline{Q}{}_{0} = \int_{ \TT^{d}} f (y) \ud y$.
    By convention, we also enforce $\underline{Q}{}_n\equiv0$ for all $ n >0$.
\end{definition}
For $ n\leqslant N$, $\underline Q{}_{-n}$ restricts to an operator $\underline
    Q{}^{(N)}_{-n}: \mC (\Lambda_{-N})\rightarrow  \mC (\Lambda_{-n})$, which we also
    denote by $\underline Q{}_{-n}$ whenever clear from context:
\begin{equs}
    \underline{Q}{}_{-n} f ( x) =L^{-  (N-n)d} \sum_{ y \in \bigsquare_{-n} ( x)\cap \Lambda_{-N}} f ( y)\,, \quad
    x \in {\Lambda}_{-n}\,.
\end{equs}
Note that $\underline{Q}{}_0$ projects onto constants,
while $\underline{Q}{}_{-N}$ restricts to the identity on $ \mC (\Lambda_{-N})$.
For convenience, we will also enforce $\underline
    Q{}^{(N)}_{-n}\equiv 0$ whenever $n>N$.

\begin{example}[Regularised white noise]\label{exmpl_white_noise}
	The averaging operator $\underline{Q}{}_{-n}$ allows us to define white noise on $\Lambda_{-n}$,
	which can be interpreted as a scale-$n$ regularisation of the original white noise $\xi$ on $\TT^{d}$.
    Namely, we define
    \begin{equation}\label{eq_def_reg_noise}
        \begin{aligned}
            \underline{\xi}{}_{n} (x) \eqdef \underline{Q}{}_{-n} \xi (x)\sim
            \mN ( 0 , L^{dn})\,, \qquad x \in \Lambda_{-n} \,,
        \end{aligned}
    \end{equation}
    which yields an independent and identically distributed family of centred
    normal random variables.
\end{example}

Having introduced the averaging operator, we can describe fluctuations across scales via
the corresponding fluctuation operators.

\begin{definition}[Fluctuation operators]
    For $n \geqslant1$, the \emph{fluctuation operator}
    $\underline{P}{}_{-n}:C^\infty(\TT^d)\rightarrow  \mC
        (\Lambda_{-n})\cap\ker(\underline Q{}_{-n+1})$ is given by
    \begin{equation*}
        \begin{aligned}
            \underline{P}{}_{-n} \eqdef  \underline{Q}{}_{-n} - \underline{Q}{}_{-n+1} \,.
        \end{aligned}
    \end{equation*}
\end{definition}

The family $(\underline{P}{}_{-n})_{1\leqslant n\leqslant N}$ allows for a decomposition
of any function $f \in \mC(\Lambda_{-N})$ into its scale components.
Namely, for any $ f \in  \mC (\Lambda_{-N}) $, we can write
\begin{equation}\label{eq:decompositionf}
    \begin{aligned}
        f = \underline{Q}{}_{0}f + \sum_{n = 1}^{N} \underline{P}{}_{-n} f\,.
    \end{aligned}
\end{equation}
In fact, these projections are orthogonal, because
\begin{equation}\label{eq:relPandQunderlined}
    \begin{aligned}
        \underline{P}{}_{-m} \underline{P}{}_{-n} = \underline{P}{}_{-m} \delta_{m , n}\,,
        \qquad \text{since} \qquad
        \underline{Q}{}_{- m} \underline{Q}{}_{-n}
        = \underline{Q}{}_{- (m \wedge n)}\,.
    \end{aligned}
\end{equation}

Rather than working directly on $ {\Lambda}_{-N}$ directly, it will be convenient to rescale
the problem to the unit lattice of size $L^N$, $\Lambda^{N} = L^N{\Lambda}_{-N} $, which 
places us in the framework of \cite{rgBook}.
We return to this point in Remark~\ref{rem:unitlattices} at the end of this section.
To carry out such rescaling, we introduce scaling operators that allow functions to be pulled back and forth between the different lattices. 

\begin{definition}[Operators on unit lattices]\label{def_unit_ops}
    For $k\in \ZZ$, we introduce the \textnormal{scaling operator}
    $S_k: \mC (\Lambda_{m}^{M})\rightarrow  \mC (\Lambda_{m-k}^{M-k})$ which acts as
    \begin{equs}
        S_k f (\bigcdot) \eqdef f(L^k\,\bigcdot)\,.
    \end{equs}
    We will be mostly interested in the cases $S_N: \mC
        (\Lambda^{N})\rightarrow  \mC (\Lambda_{-N})$ and $S_{-N}: \mC
        (\Lambda_{-N})\rightarrow  \mC (\Lambda^{N})$.

	Notably, the scaling operators allow us to scale back $\underline{Q}{}^{(N)}_{-n}$ and
    $\underline{P}{}^{(N)}_{-n}$ to $\Lambda^N$, and define
	\begin{enumerate}
		\item[(a)] $Q^{(N)}_n:  \mC
        (\Lambda^{N})\rightarrow  \mC (\Lambda_{n}^{N})$ with $$  Q^{(N)}_n\eqdef
    S_{-N}\underline{Q}{}^{(N)}_{-N+n}S_N\,,$$

\item[(b)]  $P^{(N)}_n:  \mC (\Lambda^{N})\rightarrow  \mC
        (\Lambda_{n-1}^{N})\cap\ker(Q_1S_{n-1})$ with 
	$$ P^{(N)}_n\eqdef S_{-N}\underline{P}{}^{(N)}_{-N+n-1}S_N\,.$$
	\end{enumerate}
    Again, when no ambiguity can arise, we drop the dependency on $N$ in the superscript.
\end{definition}

\begin{lemma}
	We have that $Q_{0}$ is the identity\footnote{Contrary to $ \underline{Q}$, where
	$ \underline{Q}{}_{-N}$ played the role of the identity on $ \Lambda_{-N}$.}, $Q_N$ is the projector onto constants, $P_n=Q_{n-1}-Q_n$, and $Q_n$ acts by averaging as
    \begin{equs}\label{eq_Qn_BBS}
        Q_{n} g (x)
        =
        L^{- nd} \sum_{y \in \bigsquare_n( x)\cap\Lambda^N} g(y)\,,\quad x\in\Lambda^N_n\,,
    \end{equs}
    where $\bigsquare_n(x)$ denotes unique square of side-length $L^n$ of $L^N\TT^d$
    centred at $x$.
\end{lemma}
\begin{proof}
    The first two claims are immediate.
    Next, we compute for every $x\in\Lambda^N$
    \begin{equs}
        Q_ng(x)=S_{-N}\underline{Q}{}_{-N+n}S_Ng(x)=L^{-nd}\sum_{y\in\bigsquare_{-N+n}(L^{-N}x)\cap\Lambda_{-N}}g(L^Ny)\,.
    \end{equs}
    To conclude \eqref{eq_Qn_BBS}, we notice that $y\in\bigsquare_{-N+n}(L^{-N}x)\cap\Lambda_{-N}$ is equivalent to $L^Ny\in\bigsquare_{n}(x)\cap\Lambda^{N}$.

    To prove the third claim, we compute
    \begin{equs}
        P_n=S_{-N}\underline{P}{}_{-N+n-1}S_N=S_{-N}\underline{Q}{}_{-N+n-1}S_N-S_{-N}\underline{Q}{}_{-N+n}S_N=Q_{n-1}-Q_n\,.
    \end{equs}
    This finishes the proof.
\end{proof}

\begin{example}[White noise (continued)]\label{exmpl_whitenoise_cont}
    In Example~\ref{exmpl_white_noise}, we defined the regularised white noise $
        \underline{ \xi}{}_{n}$ on the
    lattice $ \Lambda_{-n}$.
    By means of the scaling operator $ S_{-n}$ from
    Definition~\ref{def_unit_ops}, we can pushforward $ \underline{
            \xi}{}_{n}$ onto a noise on the unit lattice $ \Lambda^{n}$. Namely, we define
    \begin{equation}\label{e_def_white_noise_unit}
        \begin{aligned}
            \xi_{-n} (x)
            \eqdef L^{(-2+\alpha) n} (S_{-n} \underline{ \xi}{}_{n}) (x)\equiv L^{(-2+\alpha) n} (S_{-n} \underline{ Q}{}_{-n}\xi) (x) \sim \mN ( 0, 1)\,, \qquad x \in
            \Lambda^{n}\,,
        \end{aligned}
    \end{equation}
    which still defines an independent family of random variables. The factor $L^{(-2+\alpha) n}=
        L^{- \frac{d}{2} n}$ is the natural scaling imposed by the regularity of white
    noise, as in the Euclidean case.

    Analogously, for every $N \in \NN$ and $ n \leqslant N $, we have
    \begin{equs}\label{e_def_white_noise_unit2}
        \xi_{-n}(x) = L^{ -(2-\alpha)n} S_{N-n} Q_{N-n} S_{-N}\underline\xi{}_{N}(x) =
        L^{ (2-\alpha)(N- n)} S_{N-n} Q_{N-n} \xi_{-N}(x)\,, \quad x \in \Lambda^{n}\,,\textcolor{white}{blabl}
    \end{equs}
    where the first identity is a consequence of
    \begin{equs}
        L^{ -(2-\alpha) n}
        S_{N-n}Q_{N-n} {S_{-N}} \underline\xi{}_{N}
        &=
        L^{ -(2-\alpha) n}
        S_{-n}\underline{Q}{}^{(N)}_{-n}\underline{Q}{}_{-N}\xi\\
        &= L^{ -(2-\alpha) n}
        S_{-n}\underline{Q}{}_{-n}\xi
        =
        L^{-(2-\alpha) n}
        S_{-n}\underline{\xi}{}_{n}
        =
        \xi_{-n}\,.
    \end{equs}
    Here we used the expression of $ Q^{(N)}_{n}$ from Definition~\ref{def_unit_ops}, \eqref{eq:relPandQunderlined}, and the
    definition of $ \underline{ \xi}$ in \eqref{eq_def_reg_noise}.
\end{example}

We close this section with two useful identities.

\begin{lemma}\label{lem:Qtwice}
    For every $k,n \in\NN$, $k+n\leqslant N$, we have that
    \begin{enumerate}
        \item[(a)] $S_{k+n}Q^{(N)}_{k+n}=S_{k}Q^{(N-n)}_kS_{n}Q^{(N)}_{n}\;$,

        \item[(b)] $ S_{k+n}P^{(N)}_{k+n}=S_{k}P^{(N-n)}_kS_{n}Q^{(N)}_{n}\;$.
    \end{enumerate}
\end{lemma}
    Before we turn to the proof of this lemma, we first verify that both sides are compatible. Indeed, we have
    \begin{equs}
        \Lambda^N  \overset{Q_{n}}{\longrightarrow} \Lambda^N_{n}
        \overset{S_{n}}{\longrightarrow}\Lambda^{N-n}
        \overset{Q_{k}}{\longrightarrow}\Lambda^{N-n}_k \overset{S_{k}}{\longrightarrow}\Lambda^{N-(k+n)}\,,
    \end{equs}
    so that the right-hand side lies in $\mC (\Lambda^{N-(k+n)})$ which is the image of
    $S_{k+n}Q_{k+n}$. 
   Notice also, that Lemma~\ref{lem:Qtwice} (a) and (b)
   imply $Q_1S_{n-1}P_n=Q_1P_1S_{n-1}Q_{n-1}=0$, which justifies that the image of $
   P_{n}$ lies inside $\ker(Q_1S_{n-1})$.

  \begin{proof}[Proof of Lemma~\ref{lem:Qtwice}]
  We establish that $S_n Q_n=S_1Q_1S_{n-1}Q_{n-1}$, the first claim then follows by
  induction, and the second claim is immediate from the first one. We compute for $x\in \Lambda^N_n$
    \begin{equs}
        S_{-n+1}Q_1S_{n-1}Q_{n-1}g(x)=  L^{-dn}\sum_{y\in\bigsquare_1(L^{-n+1}x)\cap\Lambda^N}
        \sum_{z \in \bigsquare_{n-1}( L^{n-1} y)\cap\Lambda^N} g(z)=Q_ng(x)\,.
    \end{equs}
Indeed, as $y$ ranges over the square of unit mesh size and
side length $L$ centred at $L^{-n+1}x$, the rescaled points $L^{n-1}y$ range
over the square of mesh size $L^{\,n-1}$ and side length $L^{n}$ centred at $x$. 
Summing over all $z$ in the unit–mesh square of side length $L^{n-1}$ centred at $L^{n-1}y$ therefore reproduces the sum over $\bigsquare_n(x)$.
\end{proof}

\begin{remark}\label{rem:unitlattices}

We want to motivate the transition from the coarse lattices $\Lambda_{-n} \subset \TT^{d}$ to the
unit lattices $\Lambda_n$. 
Working on unit lattices guarantees that, after each RG step,
one returns to a lattice with the same fixed spacing. This renders the RG transformation homogeneous,
in the sense that the induced dynamical system is independent of explicit
occurrences of the scale~$n$.
In particular, the regularity properties of the random terms will be contained in the
spectrum of the linearisation of the RG. For instance, considering the evolution of the
noise, we see that the trajectory $(\xi_{-n})_{n\in\NN}$ of white noises on the unit lattices satisfies the recursion 
    \begin{equs}
        \xi_{-n+1}=L^{2-\alpha}S_1Q_1\xi_{-n}\,.
    \end{equs}
    Of course, this recursion is linear, and the noise blows up like $L^{2-\alpha}$ at each step, which is consistent with the fact that the regularity of the noise is $-2+\alpha-\kappa$. 
    In contrast, the trajectory $(
\underline\xi{}_{-n})_{n\in\NN}$ of white noises on the coarse lattices satisfies the
recursion
\begin{equs}
        \underline\xi{}_{n-1}=S_nQ_1S_{-n}\underline\xi{}_{n}\,.
    \end{equs}
    Here we see that the blow up $L^{2-\alpha}$ is not explicit, but hidden inside the
    fluctuations of the noise, and that the flow still depends explicitly on $n$ through the operators $S_n$ and $S_{-n}$.
    This is normal: One compares quantities that should not be compared, since they do not live on the same lattice. 
    For this reason, we will always work on the unit lattice. 
\end{remark}

\subsection{The hierarchical Laplace operator}

Finally, we introduce the hierarchical Laplace operator, which is the central
simplification in the SPDE~\eqref{e_anderson} and enables a more tractable
renormalisation analysis.

\begin{definition}[Hierarchical Laplace operator]
    The hierarchical Laplace operator $ \underline{\Delta}_{H, -N}$ on $  \mC
        (\Lambda_{-N}) $ is defined as
    \begin{equation*}
        \begin{aligned}
            - \underline{\Delta}_{H, -N}
            \eqdef
            \sum_{n =1}^{N} L^{2n} \underline{P}{}_{-n}  \,.
        \end{aligned}
    \end{equation*}
    Moreover, we define the hierarchical Laplacian $\Delta_{H,N}$ on $ \mC
        (\Lambda^{N})$ by
    rescaling of $\underline\Delta_{H,-N}$
    \begin{equs}\label{eq:defLaplacianonlambdaN}
        \Delta_{H,N}\eqdef L^{-2N}  S_{-N} \underline\Delta_{H,-N} S_N      \,.
    \end{equs}
    Whenever the value of $N$ is clear from the context, we omit the subscript, and write $\Delta_{H}$ instead of $\Delta_{H,N}$.
\end{definition}
The additional factor $ L^{-2N}$ in \eqref{eq:defLaplacianonlambdaN} arises from imposing the same scaling behaviour as that of the classical Laplace operator.
Note that when acting on $g\in  \mC (\Lambda^{N})$, the definition \eqref{eq:defLaplacianonlambdaN} reads
\begin{equs}
\Delta_{H,N}g (\bigcdot ) = L^{-2N}   \underline\Delta_{H,-N}\big(g(L^N\,\bigcdot)\big) (L^{-N}\,\bigcdot) \,.
\end{equs}
Furthermore, by explicit calculation, 
we may express $ \Delta_{H, N}$ in terms of $(P_n)_{1\leqslant n\leqslant N}$:
    \begin{equs}\label{eq_hierLaplace_BBS}
        - \Delta_{H,N}=L^{-2N}\sum_{n=1}^N L^{2n}S_{-N} \underline{P}_{-n}S_N=\sum_{n=1}^N L^{2(n-N)}P_{N-n+1}=\sum_{n=1}^N L^{-2(n-1)}P_n\,.
    \end{equs}
This representation agrees with \cite[Definition~4.1.7]{rgBook}. 

\begin{remark}
    The definition of the hierarchical Laplacian is best understood in light of~\eqref{eq:decompositionf}, which shows that the operators $\underline{P}{}_{-n}$ play the role of sharp Littlewood--Paley projections at scale $L^{-n}$.
    At a spatial scale $\ell$, the hierarchical Laplacian therefore behaves like multiplication by $\ell^{-2}$, in analogy with the usual Laplace operator.

    The operator $\underline{\Delta}_{H,-N}$ is called \emph{hierarchical} because its kernel depends on the hierarchical distance $L^{-h(x,y)}$ rather than on the Euclidean distance, where $h(x,y)$ is the largest $n\in\{0,\dots,N-1\}$ such that $x,y\in\bigsquare_{-n}(z)$ for some $z\in\Lambda_{-n}$ if $x\neq y$, with the convention $h(x,x)=\infty$.
\end{remark}

Finally, also the inverted hierarchical Laplace operator admits an explicit representation,
which we state without proof, as it may be verified by direct calculation.
\begin{lemma}\label{lem:inverse}
    Let $ g\in  \mC (\Lambda^{N})$ such that
    $ Q_{N} g =0$ (i.e. $g$ is orthogonal to constants), then
    \begin{equation}\label{e_laplace_inverse}
        \begin{aligned}
            ( - \Delta_{H,N})^{-1} g =
            \sum_{n =1}^{N} L^{2(n-1)} P_{n} g
            \,.
        \end{aligned}
    \end{equation}
    In particular, we have the useful identity
    $P_1(-\Delta_H)^{-1}=(-\Delta_H)^{-1}P_1=P_1$.
\end{lemma}

\subsection{Besov spaces}

The final technical ingredient in our framework is a notion of regularity. To this end,
we introduce hierarchical Besov spaces. We first formulate them on the torus $
\TT^{d}$, where they arise in a natural manner, and subsequently transfer the definition to functions defined on the unit lattice.

\begin{definition}[Hierarchical Besov norms]\label{def_besov}
    For any $\beta\in\RR$, we define the hierarchical Hölder--Besov space $\CB
    C^\beta\equiv\CB C^\beta(\TT^d)$ on $\TT^d$ as the completion of smooth functions
    under the norm 
    \begin{equation}\label{eq_def_besov}
        \begin{aligned}
            \| w\|_{\CB C^\beta } \eqdef
            |\underline{Q}{}_{0} w|+
            \sup_{n \in \NN} L^{\beta n}  \sup_{x\in \TT^{d}}
            |\underline{P}{}_{-n} w (x)|\,.
        \end{aligned}
    \end{equation}
\end{definition}
Indeed, recall that $\underline{P}{}_{-n}$ is the analogue of a sharp
Littlewood--Paley block at scale $L^{-n}$. In view of this, \eqref{eq_def_besov} strongly
resembles the definition of Besov spaces in the Euclidean space.
Moreover, as in the Euclidean case, whenever $\beta<0$, replacing
$\underline{P}{}_{-n}$ by $\underline{Q}{}_{-n}$ yields an equivalent
norm. Indeed, using $\underline Q{}_{-n}=\underline
    Q{}_{0}+\sum_{k=1}^n\underline P{}_{-k}$ we have
\begin{equs}
    L^{\beta n}|\underline Q{}_{-n}w(x)| &\leqslant  L^{\beta
            n}|\underline Q{}_{0}w|+\sum_{k=1}^nL^{\beta (n-k+k)} |\underline P{}_{-k}w(x)|\\
    & \leqslant \sum_{k=0}^nL^{\beta (n-k)}\| w\|_{\CB C^\beta }\leqslant (1-L^\beta)^{-1}\| w\|_{\CB C^\beta }\,.
\end{equs}
The following lemma allows us to control the Besov norm \eqref{eq_def_besov} on a coarse
lattice by that of its rescaled version on the unit lattice, where our estimates are carried out.
\begin{lemma}\label{lem:HBforlatticefunctions}
    For any $\beta\in\RR$ and $w_N\in  \mC (\Lambda_{-N})$, the Hölder--Besov norm can be more conveniently recast as
    \begin{equs}
        \| w_N\|_{\CB C^\beta } =
        |\underline{Q}{}_{0} w_N|+
        \max_{1\leqslant n\leqslant N} L^{\beta n}  \sup_{x\in \Lambda^{n}}
        |P_1S_{N-n}Q_{N-n} S_{-N}w_N (x)|\,.
    \end{equs}
\end{lemma}
\begin{proof}
    First, by definition, we have $\underline{P}{}_{-n}w_N=0 $ for $n>N$. Then, using
    Lemma~\ref{lem:Qtwice}, for $n\leqslant N$ we have $ 
        \underline{P}{}_{-n}= S_N P_{N-n+1}S_{-N}=S_{n}P_1 S_{N-n}Q_{N-n}S_{-N}$.    
\end{proof}

As an instructive example and to gain familiarity with the notation, let us
show that
white noise on $ \TT^{d}$ indeed
lies in the Besov spaces $ \CB C^{- d/2 - \kappa} $, for every $ \kappa>0 $, almost
surely.

\begin{lemma}\label{lem_white_noise}
    Let $ ({ \xi}_{-n})_{n \in \NN}$ be the family of regularised white
    noises in
    \eqref{e_def_white_noise_unit}. For every $ \kappa>0 $ and $ p \in
        (1, \infty)$, it holds
    \begin{equation*}
        \begin{aligned}
		\EE\big[  \| { \xi}\|_{\CB C^{- 2+\alpha - \kappa} (
            \TT^{d})}^{p}\big]
	    =
            \EE\Big[\Big( \sup_{ n \in \NN}L^{-\kappa n}\sup_{x\in\Lambda^n} |\xi_{-n}(x)|  \Big)^{p}\Big] \lesssim p^{p/2}\,.
        \end{aligned}
    \end{equation*}
\end{lemma}

\begin{proof}
    Let $ \kappa> 0 $.
Recall the definition of $\xi_{-n}$ on the unit lattice $\Lambda^{n}$ in terms of $\xi$,
cf. \eqref{e_def_white_noise_unit}. Since the noise has negative regularity, each occurrence of $\underline P_{-n}$ in \eqref{eq_def_besov} can be replaced by $\underline Q_{-n}$.
    Therefore, 
    \begin{equation*}
        \begin{aligned}
            \| { \xi}{}\|_{\CB C^{- 2+\alpha - \kappa} (
            \TT^{d})}=
            \sup_{n \in \NN } L^{(-2+\alpha- \kappa) n}  \sup_{x\in \TT^d}
            |\underline Q{}_{-n}\xi(x)|=\sup_{n \in \NN } L^{ -\kappa n}  \sup_{x\in \Lambda^n}
            |\xi_{-n}(x)|\,.
        \end{aligned}
    \end{equation*}
    We proceed with a classical Kolmogorov argument.
    First, for every $ p \geqslant 1$, we can crudely estimate both suprema in the previous
    display by sums, which yields
    \begin{equation*}
        \begin{aligned}
            \BF E\Big[
                \Big(
                \sup_{n \in \NN } L^{- \kappa n}  \sup_{x\in \Lambda^{n}}
                |\xi_{-n}(x)|
                \Big)^{p}
                \Big]
            \leqslant
            \sum_{n = 0}^{ \infty}
            L^{- p \kappa n}
            \sum_{x \in \Lambda^{n}}
            \BF E \Big[
            | \xi_{-n}(x) |^{p}
            \Big]\,.
        \end{aligned}
    \end{equation*}
    Next, using hypercontractivity
    of the Gaussian
    variables $ \xi_{-n} (x)$, we see that
    \begin{equation}\label{eq_supp_whitebesov}
        \begin{aligned}
            \EE\Big[
                \Big(
                \sup_{n \in \NN } L^{- \kappa n}  \sup_{x\in \Lambda^{n}}
                |\xi_{-n}(x)|
                \Big)^{p}
                \Big]
            \lesssim {p}^{p/2}
            \sum_{n = 0}^{ \infty}
            L^{- p \kappa n}
            \sum_{x \in \Lambda^{n}}
            \EE [ \xi_{-n}(x)^{2} ]^{ p /2}
            =  {p}^{p/2}
            \sum_{n = 0}^{ \infty}
            L^{(d- p \kappa) n}\,,
        \end{aligned}
    \end{equation}
    where we used $ \EE[ \xi_{-n}(x)^{2}]=1$ in the last identity.
    The right--hand side of \eqref{eq_supp_whitebesov} converges, provided we
    choose $p $ large enough. By Markov's inequality, we extend the moment bound to all
    $p \in (1, \infty) $. 
\end{proof}

\section{The renormalised model on the unit lattice}\label{sec_unit_lattice_v}

We scale the solution $ u_{N}$ to the regularised equation \eqref{e_anderson_mol} to the unit lattice $\Lambda^N$, defining a function $v_{-N}\in\Lambda^N$ by
\begin{equation*}
    \begin{aligned}
        v_{-N} (x) \eqdef L^{ \alpha N}S_{-N}u_N(x)\equiv  L^{ \alpha N} u_{ N} ( L^{-N} x)\,, \qquad x \in
        \Lambda_{N}\,,
    \end{aligned}
\end{equation*}
which now solves the equation
\begin{equation}\label{eq_supp_vN}
    \begin{aligned}
        (- \Delta_{H,N}) v_{-N}  = (L^{-\alpha N} g \xi_{-N} + L^{-2N}
        r_{N})
        v_{-N}+ L^{-(2- \alpha)N}\,.
    \end{aligned}
\end{equation}
Here, $ \xi_{-N}(x)
    \sim \mN (0,1)$ denote the
independent standard normal random variables that were defined in
\eqref{e_def_white_noise_unit} for $x\in\Lambda^N$.
For convenience, we also define
\begin{equation*}
    \begin{aligned}
        \lambda_{-N} \eqdef L^{-\alpha N} g \,, \quad
        \mu_{-N}\eqdef  L^{-2N} r_{N}\,, \quad  \text{ and } \quad \gamma_{-N}\eqdef L^{-(2- \alpha) N}\,,
    \end{aligned}
\end{equation*}
which allows us to rewrite \eqref{eq_supp_vN} as
\begin{equation}\label{eq_vN}
    \begin{aligned}
        (- \Delta_{H}) v_{-N}  = (\lambda_{-N} \xi_{-N} + \mu_{-N})
        v_{-N} + \gamma_{-N}\,.
    \end{aligned}
\end{equation}
The reader may again wonder why we consider $v_{-N}$ 
on the unit lattice instead of studying $u_N$ directly on $\TT^{d}$. 
This choice will become clear when we derive the dynamical system that describes 
how $v_{-N}$ changes across scales. We refer to Remark~\ref{rem_whyunit2} below for
further explanation.\\

Throughout this section, we analyse the SPDE \eqref{eq_supp_vN} by examining its evolution across scales. 
To this end, we introduce coarse-grained observables of $v_{-N}$, 
which allow us to capture its behaviour at progressively larger scales. 
We will conclude the section with the proof of our main result, Theorem~\ref{thm_main}, 
for which the scale analysis of $v_{-N}$ is the key ingredient.

\subsection{The coarse grained solution}

Henceforward, we enforce that $d=2$, or equivalently that $\alpha=1$. Nevertheless, we keep the parameter $\alpha$ explicit, to be able to discuss the effect of the dimension.

The central idea underlying renormalisation group techniques is to exploit
the fact that the degrees of freedom of a system depend on the scale at which it is observed. 
In what follows, we analyse the evolution of the system parameters as the reference scale is varied. 
To this end,
we define for $n\in\{0,\dots,N\}$ the average of $v_{-N}$ at scale $L^{-n}$,
which we denote by
\begin{equs}\label{eq:defofv-n}
    v_{-n}^{(N)} \eqdef L^{-\alpha(N-n)} S_{N-n}Q_{N-n} v_{-N} \,.
\end{equs}
Notice that $ v_{-n}^{(N)}$ is defined on the unit lattice of size $L^n$.
In particular,  $v_{-N}^{(N)} = v_{-N}$ and $v^{(N)}_0=\underline{Q}{}_0u_N$. 
We will again drop the superscript
and leave the dependency on the ultraviolet cutoff $N$ implicit.
The function $
    v_{-n}$ should be thought of as a coarse grained version of $ v_{-N}$,
describing the solution $ u_{N}$ of \eqref{e_anderson_mol} at scale $ L^{-n}$.
Indeed, we have
\begin{equation}\label{eq_v2u}
    \begin{aligned}
        v_{-n}^{(N)} \equiv L^{\alpha
                n}S_{N-n}Q_{N-n}S_{-N}u_N\,.
    \end{aligned}
\end{equation}
In particular, control on the trajectory $(v^{(N)}_{-n})_{0\leqslant n\leqslant N}$
will be sufficient to control  $u_N$.
To establish said control, we will rely on the fact that $ v_{-n}$ solves again a linear SPDE, the \emph{effective
    equation}, which is now driven by an \emph{effective force} at scale $n$.



    \subsection{The effective force}\label{sec_effforce}

The \emph{bare force} at scale $L^{-N}$ is given in
terms of the right--hand side of \eqref{eq_vN}, namely
\begin{equation}\label{eq_bare_force}
    \begin{aligned}
        F_{-N}[v_{-N}]\eqdef  (\lambda_{-N} \xi_{-N} + \mu_{-N})
        v_{-N} + \gamma_{-N}\,.
    \end{aligned}
\end{equation}
This allows us rewrite \eqref{eq_vN} (in integrated form) as
\begin{equs}
    v_{-N}=(-\Delta_{H,N}^{-1})F_{-N}[v_{-N}] \,.
\end{equs}
Likewise, we can introduce a family $(F_{-n})_{n\leqslant N}$ of \emph{effective
    forces} at each scale $L^{-n}$
\begin{equs}\label{eq:effectiveforce}
    F_{-n}[v_{-n}] \eqdef L^{(2-\alpha)(N-n)} S_{N-n}Q_{N-n}F_{-N}[v_{-N}]\,,
\end{equs}
such that the coarse grained solution $ v_{-n}$ satisfies the \emph{effective
    equation}
\begin{equation*}
    \begin{aligned}
        v_{-n}
        =(-\Delta_{H,n}^{-1})F_{-n}[v_{-n}]\,.
    \end{aligned}
\end{equation*}
Indeed, this is a consequence of the following chain of identities
\begin{equs}
    v_{-n}&=L^{-\alpha(N-n)}S_{N-n}Q_{N-n}v_{-N}=L^{-\alpha(N-n)}S_{N-n}Q_{N-n}(-\Delta_{H,N}^{-1})F_{-N}[v_{-N}]
    \\
    &=L^{(2-\alpha)(N-n)}(-\Delta_{H,n}^{-1})S_{N-n}Q_{N-n}F_{-N}[v_{-N}]=(-\Delta_{H,n}^{-1})F_{-n}[v_{-n}]\,,
\end{equs}
where in the third equality we used
\begin{equs}\label{eq:commutator}
\begin{aligned}
    S_{N-n}Q_{N-n}(-\Delta_{H,N}^{-1})&=
    \sum_{k=N-n+1}^NL^{2(k-1)}S_{N-n} P_k=\sum_{k=N-n+1}^NL^{2(k-1)}P_{k-N+n}S_{N-n}Q_{N-n}\\
				      &=L^{2(N-n)}\sum_{k=1}^nL^{2(k-1)}P_{k}S_{N-n}Q_{N-n}=
				   L^{2(N-n)}   (-\Delta_{H,n}^{-1})S_{N-n}Q_{N-n}\,,
                      \end{aligned}
\end{equs}
with $ S_{N-n}P_{k}^{(N)}  = P_{k
	-(N-n)}^{(n)} S_{N-n}  Q_{N-n}^{(N)}$, see Lemma~\ref{lem:Qtwice}~(b),
	being used in the second equality. Note for future reference that \eqref{eq:commutator} is equivalent to 
    \begin{equs}\label{eq:commutator1}
        \underline Q{}_{-n}(-\underline{\Delta}_{H,-N}^{-1})=(-\underline{\Delta}_{H,-n}^{-1})  \underline Q{}_{-n}\,.
    \end{equs}

The aim of the renormalisation group is to establish a closed expression for $F_{-n}$ in terms of $F_{-N}$.
In turn, this will allow the construction of $v_{-n}$ at any scale, and therefore the solution $u_N$ itself.
The difficulty with \eqref{eq:effectiveforce} lies in the fact that $F_{-N}$ is
still evaluated at $v_{-N}$,
while we only want it to depend on $ v_{-n}$.
To remove this dependence, we introduce the \emph{fluctuation field}
\begin{equation}\label{eq_def_Gamma}
	\begin{aligned}
		 \CB f_{-n}\eqdef
		 \Gamma_{N-n}F_{-N}[v_{-N}]\,, \quad \text{where } \ 
		 \Gamma_{N-n}
    \eqdef (1-Q_{N-n})(\Delta_{H,N})^{-1}
    \equiv \sum_{k =1}^{N-n} L^{2(k-1)} P_{k} \,.
	\end{aligned}
\end{equation}
The operator $ \Gamma_{N-n}$ is referred to as the fluctuation propagator.
We can then reexpress $v_{-N}$ in terms of $ v_{-n}$ and the
$\CB f_{-n}$, using \eqref{eq:decompositionf}: 
\begin{equs}\label{eq_fluctdecomp}
	v_{-N}=L^{\alpha(N-n)}S_{-N+n}v_{-n}+ \CB f_{-n}\,.
\end{equs}
Notably, the fluctuation field $ \CB f_{-n}$ can be obtained by solving a fixed
point problem driven by $F_{-N}$ only. Indeed, inserting \eqref{eq_fluctdecomp} into
\eqref{eq:effectiveforce} and \eqref{eq_def_Gamma}, we have
\begin{empheq}[left=\empheqlbrace]{alignat=2}
    \begin{aligned}\label{eq:sys}
        F_{-n}[\bigcdot]     & =L^{(2-\alpha)(N-n)}S_{N-n}Q_{N-n}F_{-N}\big[L^{\alpha(N-n)}S_{-N+n}\bigcdot+ \CB f_{-n}[\bigcdot]\big]\,, \\
        \CB f_{-n}[\bigcdot] & =\Gamma_{N-n}F_{-N}\big[L^{\alpha(N-n)}S_{-N+n}\bigcdot+ \CB f_{-n}[\bigcdot]\big]\,.
    \end{aligned}
\end{empheq}

Let us return to the explicit expression of $ F_{-N}$ in \eqref{eq_bare_force}, and pretend for
simplicity that  $$F_{-N}[v]=\lambda_{-N}\xi_{-N} v\,.$$
This is at least heuristically true, because
the terms $\mu_{-N}$ and $\gamma_{-N}$ play no essential role.
Therefore, with this simplification in \eqref{eq:sys}, we can solve explicitly for $ \CB
f_n[\bigcdot]$ in
\begin{equs}
    \big(1-\lambda_{-N}\Gamma_{N-n}(\xi_{-N}\,\bigcdot)\big)  \big[\CB
    f_{-n}\big][v_{-n}]&=L^{\alpha(N-n)}\lambda_{-N} (\Gamma_{N-n}\xi_{-N} ) S_{-N+n}v_{-n}\,,
\end{equs}
where we used that $v_{-n}$ is already averaged at scale $L^{-n}$,
and can be factorised from $\Gamma_{N-n}$.

Introducing $\lambda_{-n}\eqdef L^{\alpha(N-n)}\lambda_{-N}$, we can thus (formally)
write
\begin{equs}
    \CB f_{-n}[v_{-n}]&=\sum_{k=1}^\infty
    L^{-\alpha(N-n)(k-1)}\lambda^k_{-n}\big(\Gamma_{N-n}(\xi_{-N}\,\bigcdot)\big)^{\circ
    k} {[S_{-N+n}v_{-n}]}\,,
\end{equs}
such that the corresponding effective force \eqref{eq:effectiveforce} is given in terms of
\begin{equs}\label{eq:heuristic}
    F_{-n}[v_{-n}]&=\sum_{k=1}^\infty
    L^{(2-k\alpha)(N-n)}\lambda^k_{-n}S_{N-n}Q_{N-n}\big(\xi_{-N}\big(\Gamma_{N-n}(\xi_{-N}\,\bigcdot)\big)^{\circ
    k-1}\big){[\BF{1}]}\,  v_{-n}\,.
\end{equs}
The interest in \eqref{eq:heuristic} stems from the fact that it provides a formal
expansion of the effective force into directions that expand or contract at different
rates as $n$ decreases.
In fact, if $\alpha>0$ (i.e. $d<4$), then only a finite number of
directions expand. These contributions are called \emph{deterministically relevant}. In
$d=2$, the only relevant contribution in \eqref{eq:heuristic} beyond the noise is the term of order $\lambda_{-n}^2$. This is why, in the sequel, we will make the assumption that the effective force is of the following form:
\begin{equs}\label{eq:ansaztforce}
    F_{-n}[v_{-n}]=\big(\lambda_{-n}\xi_{-n}+\mu_{-n}+\lambda^2_{-n}\tOne_{\!\!-n}+\lambda^3_{-n}\Psi_{-n}\big)v_{-n}+\gamma_{-n}\,,
\end{equs}
where
$ \tOne_{-n}$ denotes the (centred) part of the $ \lambda_{-n}^{2}$--term in
\eqref{eq:ansaztforce},
and $\Psi_{-n}$ gathers all the \emph{deterministically irrelevant} contributions of order at least $\lambda_{-n}^3$ (or equivalently $L^{(2-3\alpha)(N-n)}$).
Indeed, in Section~\ref{sec_enhanced_noise}, we will carefully disentangle the deterministic
contribution of the inhomogeneous chaos term of order  $ \lambda_{-n}^{2}$ in
\eqref{eq:ansaztforce},
and collect it in $\mu_{-n}$.
In turn, we are only left with the second homogeneous chaos contribution
$\tOne_{\!\!-n}$.
Such an operation corresponds to the \emph{extraction}
in the context of the
renormalisation group.

Now, because $\tOne=(\tOne_{\!\!-n})_{n\leqslant N}$ is centred, its evolution along the flow
will not be dictated by the scaling of its expectation, but by that of its
variance.
When looking at the variance of $\tOne_{\!\!-n}$, we gain an entropic factor
$L^{-2(2-\alpha)(N-n)}$, which is a consequence of the fact that
\begin{equation*}
    \begin{aligned}
        \EE \big[ |Q_{1} \xi_{-n} (0)|^{2} \big]
        = L^{-2d} \sum_{x \in \bigsquare_{1} (0)} \EE \big[ | \xi_{-n} (x) |^{2}\big] = L^{-d} = L^{- 2(2-
                \alpha)}\,.
    \end{aligned}
\end{equation*}
In combination with the square of the factor $ L^{(2- 2 \alpha)(N-n)}$ in
\eqref{eq:heuristic}, this leads the variance of
$\tOne^{(N)}_{\!\! -n}$ to flow like
$L^{-2\alpha(N-n)}$. 
For this reason, we say that $\tOne$ is \emph{probabilistically irrelevant}.
We refer to the proof of Lemma~\ref{lem_stoch} for a detailed calculation.


The situation is analogous to that of the noise $ \xi$.
While the noise in \eqref{eq:heuristic} is amplified by a factor
$L^{(2-\alpha)(N-n)}$ 
along the flow, it is the covariance that
dictates its evolution. 
Indeed, we find that the covariance of the noise is constant along the flow,
due to scale invariance: A covariance calculation shows that we gain an entropic factor $L^{-d/2(N-n)}$
which exactly balances the amplification of the noise by $ L^{(2 - \alpha)
            (N-n)}$ along the flow, see also \eqref{e_def_white_noise_unit}.
The noise is therefore \emph{probabilistically marginal}.
Overall, in the subcritical regime $ \alpha> 0$, only a finite number of quantities are
deterministically relevant, and only the noise is probabilistically marginal. 

Therefore, provided expectations of deterministically relevant terms
have been placed inside the evolution of the mass term, we are left with studying
the flow of a finite collection of probabilistically irrelevant quantities through covariance computations.
Based on this heuristic,  we define rigorously $\tOne$ and
$\Psi$ in the next sections. Moreover, we show that $\tOne$ satisfies
good probabilistic estimates (see Lemma~\ref{lem_stoch}), while $\Psi$ can be
constructed deterministically given the noise $ \xi$ and $\tOne$ (see Lemma~\ref{lem_fp}).

\begin{remark}\label{rem:3_1}
	The characterisation \eqref{eq:sys} is slightly different from the equation for
    the effective force obtained in \cite[Display~(15)]{Kupiainen2016}. This is due to a fundamental difference between hierarchical and Euclidean models, related to the domain of the effective Green's function $G_\mu$.
    As Kupiainen \cite{Kupiainen2016}, we consider a fixed point problem of the form
    $v=GF[v]$, where $G$ denotes the associated Green's function. Likewise, we introduce
    scales $\mu\in L^{-\NN}$ and a trajectory a regularised Green's functions
    $(G_\mu)_\mu$ starting from $ G$, and we seek to define a trajectory of effective solutions $(v_\mu)_\mu$ starting from $v$ and solving $v_\mu=G_\mu F_\mu[v_\mu]$, where $(F_\mu)_\mu$ is a trajectory of effective forces starting from $F$.

    In the Euclidean case, we can assume that while the kernel of $G_\mu$ grows with the
    scale $\mu$, $G_\mu$ has the same domain as $G$, namely it acts by convolution on all
    (say) bounded functions on the Euclidean space. In particular, $G_\mu$ can act
    directly on $F[v]$, so a minimal choice of trajectory $(v_\mu)_\mu$ satisfying $v_0=v$ is given by $v_\mu=G_\mu F[v]$. Making this choice and recalling the definition of $F_\mu$ yields the relation $F_\mu[v_\mu]-F[v]\in\mathrm{Ker}(G_\mu)$. In fact, because shifting $F_\mu[v_\mu]$ by an element of the kernel of $G_\mu$ would not modify the equation satisfied by $v_\mu$, one has the freedom to enforce that $F_\mu[v_\mu]=F[v]$
is constant in $\mu$. 
Consequently, for the fluctuation field we have
\begin{equation*}
	\begin{aligned}
		\CB
f_\mu=v-v_\mu=GF[v]-G_\mu F_\mu[v_\mu]=(G-G_\mu)F_\mu[v_\mu]\,,
	\end{aligned}
\end{equation*}
where the right--hand side is expressed directly in terms of $F_\mu[v_\mu]$. This yields 
 the fixed point problem for the effective force:
    \begin{equs}\label{eq:floweq}
        F_\mu=F[\bigcdot+(G-G_\mu)F_\mu[\bigcdot]]\,.
    \end{equs}
    
    In the hierarchical case, the situation is different. The effective Green's function is given by $G_\mu\equiv
 (-\underline\Delta{}_{H,\log_L\mu})^{-1}$, and has domain $\mathcal
 C(\Lambda_{\log_L\mu})$, so $F[v]\in \mathcal C(\Lambda_{-N}) $ does not lie in the
 domain of $G_\mu$ for all the scales $\mu>L^{-N}$. On the other hand, in the
 hierarchical case we defined $v_\mu$ by $v_\mu\eqdef \rho_\mu v=\rho_\mu GF[v]$ where
 $\rho_\mu\equiv \underline{Q}{}_{\log_L\mu}$, cf.~\eqref{eq:defofv-n}. In view of~\eqref{eq:commutator1}, 
	$\rho_\mu$ and $G$ satisfy the commutation relation $\rho_\mu G= G_\mu\rho_\mu$,
	which implies that $v_\mu =G_\mu\rho_\mu GF[v]$. Note the presence of
	$\rho_\mu:\mathcal C(\Lambda_{-N})\rightarrow \mathcal C(\Lambda_{\log_L\mu})$,
	which brings back $F[v]$ into the domain of $G_\mu$. Because of the necessary presence of $\rho_\mu$, we only have that $F_\mu[v_\mu]=\rho_\mu F[v]$
    (in contrast to $F_\mu[v_\mu]= F[v]$ in the Euclidean case). In particular, the
    fluctuation field cannot be rewritten directly in terms of $F_\mu[v_\mu]$, since it
    also depends on $(1-\rho_\mu )F[v]$ which cannot be expressed in terms of  $F_\mu[v_\mu]$.

We observe that, with regard to analytical difficulty, the system~\eqref{eq:sys} is equivalent to the
fixed-point problem~\eqref{eq:floweq}. Indeed,~\eqref{eq:sys} has a triangular structure:
We first solve for $\CB f_{\mu}$ in terms of $F$, which is the same as solving~\eqref{eq:floweq}. Once $\CB f_{\mu}$ is known, $F_{\mu}$ is obtained directly.
\end{remark}

\begin{remark}
Yet another difference to the Euclidean RGs is that in the hierarchical model we are
restricted to the discrete set of scales $L^{- \NN}$. While in the Euclidean case, we may choose
a continuous scale $\mu\in[0,1]$, and replace \eqref{eq:floweq} by a differential equation.
Moreover, such a continuous RG flow allows to derive a
differential equation solved by the cumulants of the effective equation. This yields an inductive construction of the stochastic terms, bypassing any tedious moment computations \cite{Duch21,Duch22}.
\end{remark}

\subsection{The enhanced noise}\label{sec_enhanced_noise}

We introduced the family $ (\xi_{-n})
    _{n \in \NN}$ of independent standard normal random
variables in \eqref{e_def_white_noise_unit}.
The random field $\xi_{-n}$ is the effective white noise at scale $L^{-n}$,
see also Lemma~\ref{lem_white_noise}.
We denote by $\xi^{(N)}$ the truncation of $\xi$ at $n=N$, namely
\begin{equs}
    \xi^{(N)}_{-n}=\begin{cases}
        \xi_{-n} & \text{for }n\leqslant N\,, \\
        0        & \text{otherwise}\,.
    \end{cases}
\end{equs}
Moreover, given two linear functionals $\xi(f)$, $\xi(g)$ of the noise, we define
the Wick product
\begin{equs}
    \xi (f) \diamond  \xi (g) \eqdef \xi(f)  \xi(g) - \BF E [  \xi(f)   \xi(g)
        ]\,,
\end{equs}
to denote the projection of the product $\xi(f)  \xi(g)$ onto the second homogeneous Gaussian chaos.

Let us now turn our attention to the term $ \tOne_{\!\! - n}$ in the effective force
\eqref{eq:ansaztforce}.
For $N\in\NN$ and $0\leqslant n\leqslant N$, we define
\begin{equation}\label{eq_def_db}
    \begin{aligned}
        {\tOne}^{(N)}_{\!\! -n} \eqdef L^{(2- 2\alpha) (N-n)}
        S_{N-n}Q_{ N-n} ( \xi_{-N}\diamond \Gamma_{N-n}
        \xi_{-N})\,,
    \end{aligned}
\end{equation}
which describes the evolution of the second homogeneous chaos of the effective
force along the flow, at scale $n$. For $n<0$ or $n>N$, we enforce
${\tOne}^{(N)}_{\!\! -n}\equiv0$, by convention. Note that contrary to the
effective noise $\xi^{(N)}$, which starts from $\xi^{(N)}_{-N}=\xi_{-N}$, the
sequence $ \tOne^{(N)} =(\tOne^{(N)}_{\! \!-n})_{ n\in\NN}$ starts from $0$ at
$-N$.
Indeed, $ \tOne$ is not present in the original equation, but is only generated
by the flow.
Throughout the paper, we will often drop the dependence on the cutoff $N$ and
write $ ( \xi, \tOne )$ for the sequence $( \xi^{(N)}_{-n} , \tOne_{\!\!
        -n}^{(N)})_{n \leqslant N}$. We call this tuple the \emph{enhanced noise}.\\

For a small fixed $\kappa_{\rms}>0$, we endow the stochastic terms
$\xi=(\xi_{-n})_{n\in\NN}$ and $\tOne=(\tOne_{\!\!-n})_{n\in\NN}$ with the norm
\begin{equs}\label{eq_stoch_norm}
    \nnorm{ \xi,\tOne} \eqdef  \sup_{n\in\NN} L^{-\kappa_{\rms}
            n}\Big(\sup_{x\in\Lambda^n} | \xi_{-n}(x) |\vee\sup_{x\in\Lambda^n} | \tOne_{\!
        \!-n}(x) |^{1/2}\Big)\,.
\end{equs}
Moreover, we set
\begin{equs}\label{eq_unif_stochnorm}
    \Norm{\xi,\tOne}\eqdef\sup_{N\in\NN}\nnorm{\xi^{(N)},\tOne^{(N)}}
    \,.
\end{equs}

Throughout the rest of the article, we will assume that we are on a good event where $(\xi^{(N)},\tOne^{(N)})$ is of finite size uniformly in $N\geqslant0$.
Namely, for every
$g\in(0,1]$, we define
\begin{equs}\label{eq_def_Omega}
    \Omega_g\eqdef \big\{ \Norm{\xi,\tOne}\leqslant ( 2g)^{-1} \big\}\,.
\end{equs}
The choice to work on $\Omega_g$ is a way to eliminate the large field problem,
since occurrences of $\xi$ and $\tOne^{(N)}$ will always appear with a
$g$ and $g^2$, respectively, so that the size of these object does not pose any
problem.
The side effect is that we do not cover the full probability space, and that
the event $\Omega_g$ shrinks if one considers large $g$.
On the other hand, for small $ g$ we have an exponentially vanishing bound of the tail
probability, which is a consequence of moment bounds for \eqref{eq_stoch_norm} which we establish
in  Lemma~\ref{lem_stoch}.

\begin{lemma}\label{lem_tailprob}
    There exist constants $c, C > 0 $, such that
    for every $ g \in (0,1]$
    \begin{equs}
        \BF P(\Omega_g)\geqslant 1-Ce^{-cg^{-2}}\,.
    \end{equs}
\end{lemma}

\noindent We defer the proof of the lemma to Section~\ref{sec_stoch_est}, just after the statement of Lemma~\ref{lem_stoch}.\\

Working on events of the form $ \Omega_{g}$ is analogous to choosing a small random time
(or a small coupling constant) as in \cite{Hai14, Kupiainen2016,Duch21}.
Indeed, one could equally have chosen to randomise the coupling constant setting $ g =
    \Norm{\xi, \tOne}^{-1}$ in the present paper.

\begin{remark}
    Recall from Lemma~\ref{lem_white_noise} that the norm on the stochastic terms controls
    \begin{equs}
        \| {\xi}\|_{\CB C^{-2+\alpha-\kappa_\rms}}
        \leqslant
        \Norm{ \xi,\tOne}
        \,.
    \end{equs}
    On the other hand, for the effective second chaos $\underline\tOne{}_{\!
            n}\eqdef L^{(2-2\alpha)n} S_n\tOne{}^{(N)}_{\!- n}$ on $\Lambda_{-n}$, the
	    norm enforces that uniformly in $N$, $n$ and $x\in\Lambda_{-n}$
    \begin{equs}
        | \underline\tOne{}^{(N)}_{\! n}(x) |\lesssim
        L^{-(-2+2\alpha-2\kappa_\rms)n}\,.
    \end{equs}
    Therefore, finiteness of the stochastic norm \eqref{eq_stoch_norm} ensures
    that on the unit lattice the noise is of regularity $-2+\alpha-\kappa_\rms$,
    and the second chaos is of (better) regularity
    $-2+2\alpha-2\kappa_\rms=-2\kappa_\rms$.
    All other chaoses are of positive
    regularity, and will be constructed deterministically as functions of the
    enhanced noise.
\end{remark}



\subsection{Proof of the main result}\label{sec_eff_eq}

In the previous sections, we introduced the effective force which led to the
effective equation
\begin{equation}\label{e_v_n}
    \begin{aligned}
        (- \Delta_{H,N}) v^{(N)}_{-n} =
        \big(\lambda_{-n} \xi_{-n}
        + \lambda_{-n}^{2} \tOne^{(N)}_{\! \! -n}
        + \lambda_{-n}^{3} \Psi^{(N)}_{-n} + \mu_{-n}\big) v^{(N)}_{-n}  +
        \gamma_{-n}\,.
    \end{aligned}
\end{equation}
The right--hand side consists in particular of the enhanced noise $ ( \xi^{(N)} ,
    \tOne^{(N)})=(\xi^{(N)}_{-n},\tOne^{(N)}_{\! \!-n})_{ n\leqslant N}$,
which we can control with
high probability, see Lemma~\ref{lem_tailprob}.

Our goal in the subsequent sections is to describe the evolution of the constants
$ (\lambda_{-n}, \mu_{-n}, \gamma_{-n})_{ n \leqslant N}$, and prove that the
UV limit
\begin{equation*}
    \begin{aligned}
        ( \xi_{-n}, \tOne^{\,\star}_{\! \!-n},
        \Psi^{\star}_{- n})_{n \in \NN} \eqdef
        \lim_{N \to \infty}
        \Big( \xi^{(N)}_{-n}, \tOne^{(N)}_{\!\! -n}, \Psi^{(N)}_{-n}\Big)_{ n\leqslant N}\,,
    \end{aligned}
\end{equation*}
exists, provided we are on the good event $ \Omega_{g}$.
The main result for the trajectory of solutions $ ( v^{(N)}_{-n})_{n =0
            ,\ldots, N}$ on unit lattices,
which is the main ingredient of the proof of Theorem~\ref{thm_main},
is summarised in the following proposition.

\begin{proposition}\label{prop_main}
    Let $d=2$.
    For every $r \geqslant 1$, there exists an odd $L \in \NN$, sufficiently large
    (and depending only on $r$), such that for every $g \in (0,1]$ and every $n \in \NN$, the limit
    $
        v_{-n }^{\star} \eqdef \lim_{N \to \infty} v_{-n}^{(N)}
    $
    exists on $ \Omega_{g}$, and solves the renormalised equation
    \begin{equation*}
        \begin{aligned}
            (- \Delta_{H}) v_{-n}^{\star} =
            \big(\lambda_{-n} \xi_{-n}
            + \lambda_{-n}^{2} \tOne_{\! \! -n}^{\, \star}
            + \lambda_{-n}^{3} \Psi^{\star}_{-n} + \mu_{-n}\big)v^{\star}_{-n}  +
	    \gamma_{-n}\,,
        \end{aligned}
    \end{equation*}
    with  $ \lambda_{-n} = L^{-n} g$, $ \gamma_{-n} =
        L^{-n}$, and $ \mu_{-n} =
        L^{- 2n} \big( r - n (1- L^{-d}) g^{2} \big)$.
	
	Moreover, we have $ v^{\star}_{-n} = L^{- \alpha (m-n)}S_{m-n} Q_{m-n} v^{\star}_{-m}$, for all
	$ m>n $.
\end{proposition}

The proof of Proposition~\ref{prop_main} is the goal of Section~\ref{sec_conv}, and
will be stated at the end of said section.
Moreover, we establish uniform control over the convergence of trajectories in a Hölder
sense:

\begin{lemma}\label{lem_supp_conv_v}
	Let $d=2$.
    For every $r \geqslant 1$, there exists an odd $L \in \NN$, sufficiently large
    (and depending only on $r$), such that for every $g \in (0,1]$ and every $
    \kappa \in ( 3\kappa_{\rms},1)$ 
    \begin{equation*}
	    \begin{aligned}
		    \lim_{N \to \infty}  \sup_{n \geqslant 1} L^{- 
		    \kappa  n}  \sup_{x\in \Lambda^{n}}
	\big|P_1\big(v_{-n}^{\star}- v_{-n}^{(N)}\big)
	(x)\big| =0\,, \quad \text{ on } \Omega_{g}\,.
	    \end{aligned}
\end{equation*}
\end{lemma}


The proof of  Theorem~\ref{thm_main} is now a consequence of
Proposition~\ref{prop_main} and Lemma~\ref{lem_supp_conv_v}, by rescaling the
solution $ v_{-n}$ to the coarse grained lattice $ \Lambda_{-n} \subset
    \TT^{2}$.

\begin{proof}[Proof of Theorem~\ref{thm_main}]
	Let $ \kappa \in { (3 \kappa_{\rms}, 1)}$.
	We recall from \eqref{eq_v2u} that 
	\begin{equation*}
		\begin{aligned}
			P_1v^{(N)}_{-n}= L^{\alpha n}P_1S_{N-n}Q_{N-n}S_{-N}u_{N}\,,
		\end{aligned}
	\end{equation*}
	and that $ v^{(N)}_0=\underline Q{}_0u_{N}$. We therefore set 
	$u^{\star}_{n} \eqdef L^{ - \alpha n } S_{n} v^{\star}_{-n}$, where $
	v^{\star}$ is the limit from Proposition~\ref{prop_main}, which provides a
	candidate trajectory of renormalised solutions. In particular, we have 
	$ u^{\star}_{n} = \underline{Q}{}_{-m+n} u^{\star}_{m}$ for all $m >n$.
	We define the reconstruction  of $ (
	u^{\star}_{n})_{n \in \NN}$ to be the unique limit $ u^{\star}= \lim_{n \to
\infty} u^{\star}_{n}$. Indeed, the limit exists because the sequence is Cauchy in $ \CB C
	^{1- \kappa} ( \TT^{2})$: For $m>n$
	\begin{equation}\label{eq_supp3_main}
		\begin{aligned}
			\| u_{m}^{\star} - u_{n}^{\star}\|_{\CB C^{1 - \kappa}}
				& = 
            \sup_{n<k \leqslant m} L^{( 1- \kappa)  k}  \sup_{x\in \TT^{d}}
	    \big|\underline{P}{}_{-k} u^{\star}_{m} (x)\big|
	    = 
	    \sup_{n<k \leqslant m} L^{- \kappa  k}  \sup_{x \in \Lambda_{-k}}
	    \big|\underline{P}{}_{-k} S_{k} v^{\star}_{-k}	    (x)\big|\\
	    & \leqslant 
	    L^{- {(\kappa - 3\kappa_{\rms})} n}\Big(
	    \sup_{n<k \leqslant m} L^{- {3\kappa_{\rms}}  k}  \sup_{x \in \Lambda_{-k}}
    \big| S_{k} P_{1}^{(k)} v^{\star}_{-k}	    (x)\big|\Big)\,,
		\end{aligned}
	\end{equation}
	where used the fact that $ \underline{Q}{}_{-n} u_{m}^{\star} = u_{n}^{\star}$ in
	the first equality, 
	and in the last step that $ \underline{P}{}_{-k} S_{k} = S_{k}P_{1}^{(k)}  $, cf.
	Definition~\ref{def_unit_ops}~(b).
	The right--hand side can be chosen arbitrarily small by choosing $n$ sufficiently
	large, because the term inside the parenthesis is uniformly bounded by Lemma~\ref{lem_supp_conv_v}. Thus, $(
	u^{\star}_{n})_{n \in \NN}$ is a Cauchy sequence.

	To conclude the theorem, it is only left to show that $ u_{N}$ indeed converges to
	$u^{\star}$ on $ \Omega_{g}$, i.e. to prove that the following quantity vanishes
	in the large $N$ limit:
	\begin{equation}\label{eq_u_hoelder}
		\begin{aligned}
			\| u^{\star} - u_{N}\|_{\CB C^{ 1- \kappa} } 
			& \leqslant  
			\big\|  u^{\star} - u^{\star}_{N} \big\|_{\CB
			C^{1 - \kappa} }\
			+
			\big\| u^{\star}_{N} - u_{N}\big\|_{\CB C^{1-
			\kappa} } 
			\,.
		\end{aligned}
	\end{equation}
	The first term on the right--hand side vanishes by \eqref{eq_supp3_main}. On the other hand, 
	for the second term we apply
	Lemma~\ref{lem:HBforlatticefunctions}, which yields
	\begin{equation*}
		\begin{aligned}
			\big\| \underline{Q}{}_{-N}u^{\star} - u_{N}\big\|_{\CB C^{1-
			\kappa} } 
		&	=
			|\underline{Q}{}_{0} ( u^{\star} - u_{N})|+
        \max_{1\leqslant n\leqslant N} L^{( 1- \kappa)  n}  \sup_{x\in \Lambda^{n}}
        \big|P_1S_{N-n}Q_{N-n} S_{-N}\big( \underline{Q}{}_{-N}u^{\star} - u_{N}\big)
	(x)\big|\\
	& = 
	| v_{0}^{\star} - v_{0}^{(N)}|+
        \max_{1\leqslant n\leqslant N} L^{( 1- \kappa)  n}  \sup_{x\in \Lambda^{n}}
	\big|P_1\big( L^{- \alpha N} S_{N-n}Q_{N-n}  v_{-N}^{\star}- 
	L^{- \alpha n} v_{-n}^{(N)}\big)
	(x)\big|\\
	& = 
	| v_{0}^{\star} - v_{0}^{(N)}|+
        \max_{1\leqslant n\leqslant N} L^{- \kappa  n}  \sup_{x\in \Lambda^{n}}
	\big|P_1\big(v_{-n}^{\star}- v_{-n}^{(N)}\big)
	(x)\big|\,,
		\end{aligned}
	\end{equation*}
	where we used that $ S_{-N}\underline{Q}{}_{-N} u^{\star} = S_{-N}
	u_{N}^{\star} = L^{- \alpha N} v_{-N}^{\star}$. 
	Applying Lemma~\ref{lem_supp_conv_v} finishes the proof.
\end{proof}

\section{Derivation of the perturbative dynamical system}\label{sec_pert}

In this section, we perform the first coarse graining step, going from $
    v_{-N}$ to $ v_{-N+1}$.
We will then state a candidate for the effective equation for $
    v_{-N+1}$, using a perturbative argument.
While this derivation is non--rigorous, it sheds light on the main idea of the
Renormalisation Group mechanism.
In Section~\ref{sec_rig_RG}, we will then keep track of (and control) the
remainder terms explicitly, providing the rigorous justification for the
steps outlined here.\\

We recall from \eqref{eq:defofv-n} that
\begin{equation*}
    \begin{aligned}
        v_{-N+1}(x) 
	= L^{- \alpha} S_{1} Q_{1} v_{-N}(x) \equiv L^{-\alpha}Q_1v_{-N}(Lx)\,, \qquad x \in
        \Lambda^{N-1}\,.
    \end{aligned}
\end{equation*}
For convenience, we also introduce $ w_{-N+1} \eqdef Q_{1} v_{-N}$ such that
$ v_{-N+1}(x) = L^{- \alpha}  w_{-N+1}(Lx)$.
Moreover, we decompose the solution into 
$$ v_{-N} 
= w_{-N+1} + z_{-N+1}\,,\quad\text{with }\quad
    z_{-N+1}\eqdef P_{1} v_{-N}\,.$$
In the jargon of the Renormalisation
Group, $w_{-N+1}$ is called the \emph{background field}, since it corresponds to the
average of $v_{-N+1}$, and $z_{-N+1}$ is called the \emph{fluctuation field}. Observe that $ w_{-N+1}$ and $z_{-N+1}$ satisfy the following
useful identities
\begin{equation}\label{e_useful_w1}
    \begin{aligned}
        Q_{1} z_{-N+1} = 0 \quad \text{and}\quad
        Q_{1} ( f \, w_{-N+1} ) = (Q_{1} f ) w_{-N+1} \,, \quad \forall f \in \mC (
        \Lambda^{N})\,.
    \end{aligned}
\end{equation}
Since we are only performing a single coarse graining step below, it will be
convenient to drop the subscripts and simply write $ w = w_{-N+1}$ and $ z
    = z_{-N+1}$ for the background and fluctuation field,
respectively.

Now, by application of $ Q_{1}$ to both sides of \eqref{eq_vN}, we see that
$ w$ satisfies the equation
\begin{equation}\label{eq_supp_w1}
    \begin{aligned}
        ( - Q_{1} \Delta_{H}) w
        =
        Q_{1}\big( (\lambda_{-N}  \xi_{-N} + \mu_{-N} ) v_{-N} + \gamma_{-N}
        \big)\,,
    \end{aligned}
\end{equation}
where we used that $ -Q_{1} \Delta_{H} = \sum_{n =2}^{N} L^{-2(n-1)} P_{n} Q_{1} = (
    -Q_{1} \Delta_{H}) Q_{1} $.
Using the decomposition $ v_{-N} = w +z $ and \eqref{e_useful_w1}, the right--hand side
can then be cast into
\begin{equation}\label{eq_supp_w2}
    \begin{aligned}
        ( - Q_{1} \Delta_{H}) w
        =
        \lambda_{-N}( Q_{1} \xi_{-N}) w
        + \lambda_{-N} Q_{ 1} ( \xi_{-N} z)
        + \mu_{-N} w + \gamma_{-N}
        \,.
    \end{aligned}
\end{equation}
In particular, the right--hand side is expressed entirely in terms of $w$,
except for the second term which depends on $ z $.
In order to replace $ z$ with a suitable expression in terms of $ w$,
we write the analogue of \eqref{eq_supp_w2} for the fluctuation field $ z$.
To this end, we apply
$ P_{1}$ to both sides of \eqref{eq_vN}, which yields
\begin{equation}\label{eq_z}
    \begin{aligned}
        z & =
        P_{1}\big( (\lambda_{-N}  \xi_{-N} + \mu_{-N} ) ( w+ z) +
        \gamma_{-N}\big) \\
          & =
        \lambda_{-N} P_{1} ( \xi_{-N} z)
        +
        \lambda_{-N} ( P_{1} \xi_{-N}) w
        +
        \mu_{-N} z\,,
    \end{aligned}
\end{equation}
where we used that $P_1(-\Delta_H)^{-1}=P_1$, see Lemma~\ref{lem:inverse}, and in the second step that $ P_{1} $ projects constants onto zero, and that $
    P_{1} w = 0 $ because $ P_{1} Q_{1} = 0 $.
Next, we notice that  $ w
    \mapsto  z[w]$ is
linear, because we were considering an affine equation. Therefore, it is convenient to
introduce the random field $\mM$ such that $ z = \mM w$, which then satisfies the equation
\begin{equation}\label{eq_M}
    \begin{aligned}
        \mM  =
        \lambda_{-N} P_{1} ( \xi_{-N} \mM)
        +
        \lambda_{-N} ( P_{1} \xi_{-N})
        +
        \mu_{-N} \mM\,,
    \end{aligned}
\end{equation}
which is independent of $ w$. Notice that the independence of $w$ is a consequence of the fact
that the equation we are considering is linear. Thus, in order to stay close to
the core idea of RG, we will not attempt to solve for $\mM$ directly.

Overall, we have successfully derived two coupled equations

\begin{empheq}[left=\empheqlbrace]{alignat=2}\begin{aligned}\label{eq:sysRG_firststep}
        ( - Q_{1} \Delta_{H}) w
            & =
        \lambda_{-N}( Q_{1} \xi_{-N}) w
        + \lambda_{-N} Q_{ 1} ( \xi_{-N} \mM ) w
        + \mu_{-N} w + \gamma_{-N}\,,
        \\
        \mM & =
        \lambda_{-N} P_{1} ( \xi_{-N} \mM)
        +
        \lambda_{-N} ( P_{1} \xi_{-N})
        +
        \mu_{-N} \mM\,,
    \end{aligned}
\end{empheq}
describing the coarse grained behaviour of $ v_{-N}$ at scale $ L^{-N+1}$, as well as the
fluctuation behaviour $\mM$ that was lost when mapping from $ v_{-N}$ to $w$.

By formally expanding the quantities in the system \eqref{eq:sysRG_firststep} in the coupling constant $\lambda_{-N}$, we expect that to leading order $\mM $ is well
approximated by
\begin{equation*}
    \begin{aligned}
        \mM^{[1]}\eqdef \lambda_{-N} ( P_{1} \xi_{-N})\,,
    \end{aligned}
\end{equation*}
which forms the basis of the perturbative argument that is about to follow.
Here, the superscript indicates that we only approximate using terms in a single
power of $ \lambda_{N}$.
We then expect that up to lower--order corrections of order $\mO (
    \lambda_{-N}^{3}) $, $ w $ is well described
in terms of $ w^{[2]}$, which solves the
equation
\begin{equation}\label{e_pert_w1}
    \begin{aligned}
        ( - Q_{1} \Delta_{H}) w^{[2]}
        =
        \lambda_{-N}( Q_{1} \xi_{-N}) w^{[2]}
        + \lambda_{-N}^{2} Q_{ 1} ( \xi_{-N} P_{1}
        \xi_{-N}) w^{[2]}
        + \mu_{-N} w^{[2]} + \gamma_{-N}\,,
    \end{aligned}
\end{equation}
where we simply replaced $ \mM $ in \eqref{eq:sysRG_firststep} with $ \mM^{[1]}$.
In other words, we perform a perturbative expansion up to order $
    \lambda_{-N}^{2}$, while neglecting all remaining terms that are of order $\mO
    ( \lambda_{-N}^{3}) $.\\

We have found an approximate candidate equation for $ w$ which can
now be scaled into an effective equation on $ \Lambda^{N-1}$ using $ v_{-N+1}^{[2]}\eqdef L^{- \alpha} S_{1}
    w^{[2]}$:
\begin{equs}\label{e_v1_eff_wFluc}
    ( - \Delta_{H}) v_{-N+1}^{[2]}
    &=
    L^{2- \alpha}S_{1}\Big(
    \big(
    \lambda_{-N}( Q_{1} \xi_{-N})
    + \lambda_{-N}^{2} Q_{ 1} ( \xi_{-N} P_{1}
    \xi_{-N})
    + \mu_{-N}\big) w^{[2]}
    + \gamma_{-N}\Big)
    \\
    & =
    \big(
    L^{\alpha}\lambda_{-N} \xi_{-N+1}
    + L^{2}\lambda_{-N}^{2}  S_{1}
    Q_{ 1} ( \xi_{-N} P_{1}
    \xi_{-N})
    + L^{2}\mu_{-N}  \big)  v_{-N+1}^{[2]}
    + L^{2- \alpha}\gamma_{-N} \,,
\end{equs}
where we used the definition of $ \xi_{-N+1}$ in
\eqref{e_def_white_noise_unit2}, the fact that
$ S_{1}(f \cdot g) =( S_{1}f ) ( S_{1} g) $ for all $f, g \in \mC (
    \Lambda^{N})$, and moreover that for every $ f \in
    \Lambda^{N}_{1}$
    \begin{equation*}
    \begin{aligned}
        S_{1}Q_{1}^{(N)}(- \Delta_{H,N}) f =
        \sum_{n =2}^{N} L^{-2(n-1)} S_{1} P_{n}^{(N)} f
        =
        L^{-2}
        \sum_{n =1}^{N-1} L^{-2(n-1)} P_{n}^{(N-1)} S_{1}f  = L^{-2} (-\Delta_{H, N-1})
        S_{1}f \,,
    \end{aligned}
\end{equation*}
which gives rise to the extra factor $ L^{2}$.
Next, we decompose the second inhomogeneous chaos term on the right--hand side
of \eqref{e_v1_eff_wFluc}
into
\begin{equation}\label{eq_vN-1_supp}
    \begin{aligned}
        L^{2}\lambda_{-N}^{2}  S_{1}
        Q_{ 1} ( \xi_{-N} P_{1}
        \xi_{-N})
         & =
        L^{2\alpha}\lambda_{-N}^{2}
        \tOne_{\! \!-N+1}
        +
        L^{2}\lambda_{-N}^{2}
        \BF E\big[
            S_{1}
            Q_{ 1} ( \xi_{-N} P_{1}
            \xi_{-N})
        \big] \\
         & =
        L^{2\alpha}\lambda_{-N}^{2}
        \tOne_{\! \!-N+1}
        +
        L^{2} \lambda_{-N}^{2} ( 1- L^{-d}) \,,
    \end{aligned}
\end{equation}
where we used \eqref{eq_def_db} and Lemma~\ref{lem_expect_db}. Hence,
\eqref{e_v1_eff_wFluc} can be brought into
the form \eqref{e_v_n}, such that $ v_{-N+1}$ approximately solves the
equation
\begin{equation}\label{eq_vN-1}
    \begin{aligned}
        (- \Delta_{H}) v_{-N+1}^{[2]}
        =
        \big(\lambda_{-N+1} \xi_{-N+1}
        + \lambda_{-N+1}^{2} \tOne_{\! \! -N+1}
        + \mu_{-N+1}\big) v_{-N+1}^{[2]}
        + \gamma_{-N+1}\,,
    \end{aligned}
\end{equation}
with
\begin{empheq}[left=\empheqlbrace]{alignat=2}
    \begin{aligned}\nonumber
        \lambda_{-N+1} & = L^{\alpha} \lambda_{-N}\,,   \\
        \mu_{-N+1}     & = L^{2} \big(\mu_{-N}
        + (1- L^{-d}) \lambda_{-N}^{2}
        \big) \,,                                       \\
        \gamma_{-N+1}  & = L^{2- \alpha} \gamma_{-N}\,,\\
	\tOne_{\! \!-N+1} & = L^{2 - 2 \alpha}S_{1}
        Q_{ 1} ( \xi_{-N}\diamond P_{1} \xi_{-N})\,.
    \end{aligned}
\end{empheq}
In other words, the coefficients of the effective equation for $ v_{-N+1}^{[2]}$ can
be entirely expressed in terms of the coefficients of the equation for $
    v_{-N}$, on the lattice $ \Lambda^{N}$.\\

Thus far, we have only performed the first step of the renormalisation group
procedure. In the same manner, we can now go from $ v_{-N+1}^{[2]}$ to a $
    v_{-N+2}^{[2]}$.
The only difference to the first perturbative step is that we have created
a new noise term, namely $ \tOne_{\!\! -(N+1)}$, which in turn creates new
terms of order $ \mO ( \lambda_{-N+1}^{3})$. Again, we may ignore
those due to our current perturbative perspective.
More generally, this algorithm allows to map $
    v_{-n}^{[2]}$ to $ v_{-n+1}^{[2]}$, for any fixed $n \leqslant N$. 
The flow of force coefficients in the effective equation 
\begin{equation*}
	\begin{aligned}
		(- \Delta_{H}) v_{-n}^{[2]}
        =
        \big(\lambda_{-n} \xi_{-n}
        + \lambda_{-n}^{2} \tOne_{\! \! -n}
        + \mu_{-n}\big) v_{-n}^{[2]}
        + \gamma_{-n}\,,
	\end{aligned}
\end{equation*}is then expressed in terms of the dynamical system 
\begin{empheq}[left=\empheqlbrace]{alignat=2}
    \begin{aligned}\label{eq:sys_a}
        \lambda_{-n+1}    & = L^{\alpha} \lambda_{-n}\,,                        \\
        \mu_{-n+1}        & = L^{2} \big(\mu_{-n}
        + (1- L^{-d}) \lambda_{-n}^{2}
        \big) \,,                                                               \\
        \gamma_{-n+1}     & = L^{2- \alpha} \gamma_{-n}\,,                      \\
        \tOne_{\!\! -n+1} & = L^{2- 2 \alpha}\big( S_{1} Q_{1}\tOne_{\!\! -n} +
        S_{1}
        Q_{ 1} ( \xi_{-n} \diamond P_{1}
        \xi_{-n})\big)
        \,.
    \end{aligned}
\end{empheq}
Notice that \eqref{eq:sys_a} provides an alternative, recursive representation of $ \tOne$, which indeed
agrees with \eqref{eq_def_db}, see Lemma~\ref{lem_db_flow_same} below.\\

Overall, we have derived a recursive characterisation of the effective force coefficients
in \eqref{e_v_n}, albeit perturbative at this stage.
We recall from Section~\ref{sec_effforce} that $\lambda=(\lambda_{-n})_{ n\leqslant N}$, $\mu=(\mu_{-n})_{
            n\leqslant N}$ and $\gamma=(\gamma_{-n})_{ n\leqslant N}$
are (deterministically) relevant quantities, since they increase along the flow.
To make sure that these relevant terms stay bounded along the RG trajectory, it
is necessary to solve their flows backwards, by prescribing a finite value at $n=0$.
Hence, we enforce that the renormalised couplings are given in terms of
\begin{equation*}
	\begin{aligned}
    \lambda_0=g\in(0,1]\,,\quad \mu_0=r\geqslant1\,,\quad\text{and}\quad\gamma_0=1\,.		
	\end{aligned}
\end{equation*}
Solving then the flow equation \eqref{eq:sys_a} yields
\begin{equation}\label{eq_pert_para}
    \begin{aligned}
        \lambda_{-n} = L^{- \alpha n} g\,, \quad\mu_{-n} = L^{- 2n} \big( r - (1- L^{-d}) n \,g^{2} \big)\quad \text{and } \quad
        \gamma_{-n} = L^{- (2 - \alpha) n} \,,
    \end{aligned}
\end{equation}
which is consistent with the bare equation \eqref{e_anderson_mol}. In
particular, we recover the bare mass $ r_{N} = L^{2N} \mu_{-N}$.

On the other hand, the probabilistically irrelevant term $(\tOne_{\!\! - n})_{n \leqslant N}$
will be created by the flow, and can be proven to remain bounded with high probability. In fact, we will show
that $ \tOne_{\!\! - n}^{(N)}$ converges in the UV limit $N \to \infty$. 
In turn, convergence of both the deterministic and random force coefficients allows us to
conclude convergence of $ v_{-n}^{(N),[2]}$, as $ N \to \infty$.
We thus formally recover the statement of Proposition~\ref{prop_main} with $ \Psi_{-n}^{\star} \equiv 0$.
The remainder of the paper makes this argument rigorous and provides all the necessary details,
in particular the control of the remainder term.\\

We conclude the section by showing that the recursive expression of $ \tOne $
is compatible with~\eqref{eq_def_db}.

\begin{lemma}\label{lem_db_flow_same}
    For every $N \in \NN$, and $ n \leqslant N $, we have
    \begin{equs}\label{eq:reccov2}
	    \tOne_{\! \!-n}^{(N)}=\sum_{k=1}^{N-n}L^{(2-2\alpha)k}S_kQ_k(\xi_{-n-k}\diamond P_1\xi_{-n-k})\,.
    \end{equs}

\end{lemma}

\begin{proof}
    We verify the recursion relation using the definition of $\tOne $ from
    \eqref{eq_def_db}. First, we notice that
    \begin{equation*}
	    \begin{aligned}
		   {\tOne}_{\!\! -n+1}                                                               
	     &= L^{(2- 2\alpha) (N-n+1)}
            S_{N-n+1}Q_{ N-n+1} ( \xi_{-N}\diamond \Gamma_{N-n+1}
            \xi_{-N})                                                                            \\
             & =L^{(2- 2\alpha)}L^{(2- 2\alpha) (N-n)}
            S_1Q_1 S_{N-n}Q_{ N-n} \big( \xi_{-N}\diamond (\Gamma_{N-n}+L^{2(N-n)} P_{N-n+1})
            \xi_{-N}\big)  \,,
	    \end{aligned}
    \end{equation*}
where we used Lemma~\ref{lem:Qtwice}~(a), and the definition of
$\Gamma_{N-n+1}$ \eqref{eq_def_Gamma}. 
The first term on the right--hand side recovers ${\tOne}_{\!\! -n}$. 
Next,
we apply Lemma~\ref{lem:Qtwice}~(b) which yields
    \begin{equation*}
	    \begin{aligned}
		     {\tOne}_{\!\! -n+1}                                                               
		    & =L^{(2- 2\alpha)}S_1Q_1{\tOne}_{\!\! -n}+L^{(2- 2\alpha)}L^{(4- 2\alpha) (N-n)}
            S_1Q_1 S_{N-n}Q_{ N-n} ( \xi_{-N}\diamond S_{-N+n} P_{1}S_{N-n}Q_{N-n}
            \xi_{-N})
            \\
             & =L^{(2- 2\alpha)}S_1Q_1{\tOne}_{\!\! -n}+L^{(2- 2\alpha)}L^{(4- 2\alpha) (N-n)}
            S_1Q_1  ( S_{N-n}Q_{ N-n}\xi_{-N}\diamond  P_{1}S_{N-n}Q_{N-n}
            \xi_{-N})
            \\
             & =L^{(2- 2\alpha)}S_1Q_1{\tOne}_{\!\! -n}+L^{(2- 2\alpha)}
            S_1Q_1  ( \xi_{-n}\diamond  P_{1}
            \xi_{-n})\,.
	    \end{aligned}
    \end{equation*}
where in the second identity we used that $S_{-N+n}P_1S_{N-n}Q_{N-n}\xi_{-N}$ is already averaged at
scale $N-n$ and can factorise from $Q_{N-n}$. The last identity is a
consequence of  the definition of $\xi_{-n}$ \eqref{e_def_white_noise_unit2}.
 We finish the proof, by noticing that the recursion in \eqref{eq:sys_a} yields precisely the
    right--hand side of \eqref{eq:reccov2}.
   \end{proof}

\begin{remark}[A single large RG step]\label{rem_why_not_1jump}
    The reader may have spotted that the chain of coarse grainings can be
    performed in a single (large) step using $ v_{-n}=
        L^{ - \alpha (N-n)} S_{N-n}Q_{N-n} v_{-N} $
    by its definition in~\eqref{eq:defofv-n}:
    \begin{equation*}
        \begin{aligned}
            \begin{tikzpicture}[baseline=(current bounding box.center)]
                \node (vN) at (1,0.0) {$v_{-N}$};
                \node (vN1) at (2,1) {$v_{-N+1}$};
                \node (vN2) at (4,1) {$v_{-N+2}$};
                \node (vdots) at (6,1) {$\cdots$};
                \node (vn1) at (8,1) {$v_{-n-1}$};
                \node (vn) at (9,0.0) {$v_{-n}$};

                \draw[->] (vN) -- (vN1);
                \draw[->] (vN1) -- (vN2);
                \draw[->] (vN2) -- (vdots);
                \draw[->] (vdots) -- (vn1);
                \draw[->] (vn1) -- (vn);

                \draw[->] (vN) to (vn);
            \end{tikzpicture}
        \end{aligned}
    \end{equation*}
    Indeed, the same perturbative dynamical system \eqref{eq:sys_a}
    can be obtained by performing a single large RG step.
    However, we will leverage on the step--by--step procedure, which is central
    to the RG framework, when controlling the remainder term.
    This will be the focus of Section~\ref{sec_rig_RG}.
\end{remark}

\section{Uniform control of the RG flow}\label{sec_rig_RG}

In Section~\ref{sec_pert}, we derived perturbatively an effective equation for
$ v_{-n} = v_{-n}^{[2]}+ \mO ( \lambda_{-n}^{3})$, of the form
\begin{equation*}
    \begin{aligned}
        (- \Delta_{H}) v_{-n}^{[2]}
        =
        \big(\lambda_{-n} \xi_{-n}
        + \lambda_{-n}^{2} \tOne_{\! \! -n}
        + \mu_{-n}\big) v_{-n}^{[2]}
        + \gamma_{-n}\,.
    \end{aligned}
\end{equation*}
However, we entirely neglected control of the remainder term $
    \lambda_{-n}^{3} \Psi_{-n}=  \mO (
    \lambda_{-n}^{3})$.

The goal of the present section is to expand our perturbative analysis to
control this error.
We then establish the required stochastic estimates for the term $ \tOne$.
Finally, we present a fixed point argument which yields uniform (in $N$) control of the remainder
$ \Psi^{(N)}=(\Psi^{(N)}_{-n})_{0\leqslant n \leqslant N}$,
and consequently control of the trajectory $ v^{(N)}=(v^{(N)}_{-n})_{0\leqslant n \leqslant N}$.

\subsection{The dynamical system}\label{sec_rg_general}

In Section~\ref{sec_pert}, we only performed the first (perturbative) RG step,
and then deduced the dynamical system for the subsequent steps.
In the following, we will not treat the first step differently, instead we
present generally how to derive  the effective
equation for $v_{-n+1} = L^{-\alpha} S_{1} Q_{1} v_{-n}$ from the effective equation for $ v_{-n}$, for $1 \leqslant n \leqslant N$.

The basis of the derivation is that $
    v_{-n}= v_{-n}^{(N)}$ solves equation \eqref{e_v_n}:
\begin{equation*}
    \begin{aligned}
        (- \Delta_{H} ) v_{-n}
        = (\lambda_{-n} + \lambda_{-n}^{ 2} \tOne_{\!\! -n} + \lambda_{-n}^{3}
        \Psi_{-n}+ \mu_{-n}) v_{-n} + \gamma_{-n}\,,
    \end{aligned}
\end{equation*}
Clearly, this holds true when $n =N$, with $
    \tOne_{-N}, \Psi_{-N} \equiv 0$, cf. \eqref{eq_vN}.
As in Section~\ref{sec_pert}, we write $ v_{-n} = w_{-n+1} + z_{-n+1}$, where
\begin{equation*}
    \begin{aligned}
        w_{-n+1} \eqdef Q_{1} v_{-n}\,, \qquad \text{and} \qquad z_{-n +1} \eqdef P_{1}
        v_{-n}\,.
    \end{aligned}
\end{equation*}
Then, analogously to \eqref{eq_supp_w2}, we see that $ w_{-n+1}$ solves the
equation
\begin{equation}\label{q_supp_w3}
    \begin{aligned}
        ( - Q_{1} \Delta_{H}) w_{-n+1}
         & =
        \big(\lambda_{-n}  Q_{1} \xi_{-n} + \lambda_{-n}^{2} Q_{1} \tOne_{\!\! -n}
        +\lambda_{-n}^{3} Q_{1} \Psi_{-n} + \mu_{-n} \big) w_{-n+1}
        + \gamma_{-n} \\
         & \quad+
        \lambda_{-n} Q_{1}\big( \xi_{-n} z_{-n+1}\big)
        +
        \lambda_{-n}^{2} Q_{1}( \tOne_{\!\! -n} z_{-n+1})
        +
        \lambda_{-n}^{3} Q_{1}( \Psi_{-n} z_{-n+1})
        \,.
    \end{aligned}
\end{equation}
Furthermore, analogously to \eqref{eq_z}, $ z_{-n+1}$ solves the equation
\begin{equation}\label{eq_supp2_z}
    \begin{aligned}
        z_{-n +1}
         & =
        \big(\lambda_{-n}  P_{1} \xi_{-n} + \lambda_{-n}^{2} P_{1} \tOne_{\!\! -n}
        +\lambda_{-n}^{3} P_{1} \Psi_{-n} \big) w_{-n+1} \\
         & \quad+
        \lambda_{-n} P_{1}\big( \xi_{-n} z_{-n+1}\big)
        +
        \lambda_{-n}^{2} P_{1}( \tOne_{\!\! -n} z_{-n+1})
        +
        \lambda_{-n}^{3} P_{1}( \Psi_{-n} z_{-n+1})     + \mu_{-n} z_{-n+1}\,.
    \end{aligned}
\end{equation}
Here, we used again that $ Q_{1} z_{-n+1}=0$, $ P_{1} \gamma_{-n}=0$, and $ P_{1} w_{-n +1}=0$.
Once more, we notice that $ w \mapsto z_{-n+1}[w]$ is linear, and introduce the
random field $ \mM_{-n+1}$ such that $ z_{-n+1}= \mM_{-n+1} w_{-n+1}$.
Replacing $ z_{-n+1}$ with this new representation, \eqref{eq_supp2_z} implies
\begin{equation}\label{e_supp2_M}
    \begin{aligned}
        \mM_{-n +1}
         & =
        \lambda_{-n}  P_{1} \xi_{-n} + \lambda_{-n}^{2} P_{1} \tOne_{\!\! -n}
        +\lambda_{-n}^{3} P_{1} \Psi_{-n}             \\
         & \quad+
        \lambda_{-n} P_{1}\big( \xi_{-n} \mM_{-n+1}\big)
        +
        \lambda_{-n}^{2} P_{1}( \tOne_{\!\! -n} \mM_{-n+1})
        +
        \lambda_{-n}^{3} P_{1}( \Psi_{-n} \mM_{-n+1}) +\mu_{-n}  \mM_{-n+1}\,.
    \end{aligned}
\end{equation}
Combining \eqref{q_supp_w3} and \eqref{e_supp2_M} yields the non--perturbative
equivalent of the coupled system \eqref{eq:sysRG_firststep}.

Next, we proceed differently than in the perturbative derivation. We define the
remainder
\begin{equation}\label{eq_def_rho}
    \begin{aligned}
        \lambda_{-n}^{2} \mR_{-n+1}
        \eqdef
        \mM_{-n+1} - \lambda_{-n} P_{1} \xi_{-n}\,,
    \end{aligned}
\end{equation}
which we neglected in Section~\ref{sec_pert}. Combining \eqref{e_supp2_M} with the decomposition \eqref{eq_def_rho} and rearranging terms, we
derive the following equation for $ \mR_{-n+1}$:
\begin{equation}\label{e_supp1_rho}
    \begin{aligned}
        \mR_{-n+1}
         & =
        P_{1} \tOne_{\!\! -n}
        + P_{1}(  \xi_{-n} (P_{1} \xi_{-n}))                  \\
         & \quad
        +\lambda_{-n} P_{1} \Psi_{-n}
        +
        \lambda_{-n} P_{1}\big( \xi_{-n}\mR_{-n+1})
        + \lambda_{-n} P_{1 }(\tOne_{\!\! -n}  P_{1}\xi_{-n}) \\
         & \quad
        +
        \lambda_{-n}^{2} P_{1}(\tOne_{\!\! -n}  \mR_{-n+1})
        +
        \lambda_{-n}^{ 2}  P_{1}( \Psi_{-n}  P_{1} \xi_{-n})
        \\
         & \quad+
        \lambda_{-n}^{3}  P_{1}( \Psi_{-n} \mR_{-n+1})
        + \mu_{-n} \lambda_{-n}^{-1} P_{1} \xi_{-n}
        + \mu_{-n} \mR_{-n+1}\,.
    \end{aligned}
\end{equation}
Finally, inserting $ z_{-n+1}= (\lambda_{-n} P_{1} \xi_{-n} +
    \lambda^{2}_{-n} \mR_{-n+1}) w_{-n+1}$ into \eqref{q_supp_w3}, we can write the effective equation for
$ v_{-n+1}= L^{- \alpha} S_{1} w_{-n+1} $ in terms of
\begin{equation}\label{e_eff_vnplus}
    \begin{aligned}
        (- \Delta_{H} ) v_{-n+1}
        = (\lambda_{-n+1} \xi_{-n+1} + \lambda_{-n+1}^{ 2} \tOne_{\!\! -n+1} + \lambda_{-n+1}^{3}
        \Psi_{-n+1}+ \mu_{-n+1}) v_{-n+1} + \gamma_{-n+1}\,,
    \end{aligned}
\end{equation}
with the following regrouping of terms
\begin{empheq}[left=\empheqlbrace]{alignat=2}
    \begin{aligned}\label{e_rg_system}
        \lambda_{-n+1}    & = L^{\alpha} \lambda_{-n}\,,                                             \\
        \mu_{-n+1}        & = L^{2} \big(\mu_{-n}
        + (1- L^{-d}) \lambda_{-n}^{2}
        \big) \,,                                                                                    \\
        \gamma_{-n+1}     & = L^{2- \alpha} \gamma_{-n}\,,                                           \\
  \xi_{-n+1}&=L^{2-\alpha}S_1Q_1\xi_{-n}\,,
        \\
        \tOne_{\!\! -n+1} & = L^{2- 2 \alpha}\big( S_{1} Q_{1}\tOne_{\!\! -n} +
        S_{1}
        Q_{ 1} ( \xi_{-n} \diamond P_{1}
        \xi_{-n})\big)\,,                                                                            \\
        \Psi_{-n+1}
                          & = L^{2- 3 \alpha} \big(
        S_{1}Q_{1} \Psi_{-n}
        +
        S_{1}Q_{1}( \xi_{-n} \mR_{-n+1})
        + S_{1}Q_{1}( \tOne_{\!\! -n} P_{1} \xi_{-n})                                                \\
                          & \hspace{1.6cm}  + \lambda_{-n} S_{1} Q_{1}(\tOne_{\!\! -n} \mR_{-n+1}) +
        \lambda_{-n} S_{1}Q_{1}( \Psi_{-n} P_{1} \xi_{-n})                                           \\
                          & \hspace{1.6cm} + \lambda_{-n}^{2} S_{1} Q_{1}( \Psi_{-n} \mR_{-n+1})
        \big)
        \,.
    \end{aligned}
\end{empheq}
Notice that the above dynamical system is an extension of the one we derived
perturbatively in
\eqref{eq:sys_a}, now incorporating the higher-order error term represented by
$\Psi$. 
Moreover, we observe that the linear evolution of the remainder term $\Psi$ is given
by multiplication with $L^{2-3\alpha}=L^{-1}$. 
This renders the heuristic discussion in Section~\ref{sec_pert} precise and confirms that
$\Psi$ is indeed deterministically irrelevant.
As will be shown in the next section, this stands in clear contrast to the term $\tOne$,
which is only  probabilistically irrelevant. Indeed, given $(\xi,\tOne)$, the
good prefactor $L^{2- 3 \alpha}$ allows us to construct $\Psi$ by deterministic means.

\begin{remark}[Why the unit lattice?]\label{rem_whyunit2}
    Let us return to the question why we view the evolution of scales on the unit
    lattice, rather than the equivalent trajectory $(u^{(N)}_{n})_{0\leqslant
    n\leqslant N}$ of effective solutions defined directly on lattices $\Lambda_{-n}
    \subset \TT^{d}$, with $u_n^{(N)}=\underline Q{}^{(N)}_{-n}u_N$.

     Notice that the system~\eqref{e_rg_system} is particularly transparent, as the spectrum of its linearisation reflects directly the regularity properties of the terms involved.
     It particular, it is immediate from \eqref{e_rg_system} that the remainder $\Psi$
     contracts along the flow, since the linearisation of its evolution is governed by
     the factor $L^{2-3\alpha}<1$. 
    In contrast, if we were to work with the trajectory $(u^{(N)}_{n})_{0\leqslant
    n\leqslant N}$, the situation would be different. In that case, the spectrum of the
    linearisation of the RG would be trivial, and the scaling properties of the different
    terms would be hidden within their definitions (see Remark~\ref{rem:unitlattices} for a discussion of this phenomenon for the evolution of the noise). 
    In other words, rescaling to the unit lattice at each step allows for direct
    comparison of the different quantities.
\end{remark}

\subsection{Stochastic estimates}\label{sec_stoch_est}

In this section, we prove the stochastic estimate of the enhanced noise $ (
    \xi^{(N)},\tOne^{(N)}) = \big( \xi^{(N)}_{-n }, \tOne_{\!
        \!-n}^{(N)}\big)_{ n \leqslant N}$.
Namely, we prove the following lemma:

\begin{lemma}[Stochastic estimate]\label{lem_stoch}
    For every $p \in (1, \infty) $, it holds that
    \begin{equs}\label{eq:L^pcoupling}
        \BF E \big[\Norm{\xi,\tOne}^p\big]
        =
        \BF E \Big[ \sup_{N \in \NN} \nnorm{\xi^{(N)},\tOne^{(N)}}{}^p\Big]\lesssim
        p^{p/2} \,,
    \end{equs}
    with an implicit constant that is independent of $p$.
\end{lemma}

In fact, we control $( \xi, \tOne)$ in the full subcritical regime $ \alpha>0$, which includes $
    d=2,3$. Therefore, we keep the variable $ \alpha$ explicit throughout the
proof.

\begin{proof}
	We begin by noticing from \eqref{eq_unif_stochnorm} that
    \begin{equation*}
        \begin{aligned}
            \Norm{  \xi,\tOne }
            = \| \xi \|_{\CB C^{- 2+\alpha - \kappa_{\rms}} (
            \TT^{d})} \vee
            \sup_{N\in\NN} \max_{0\leqslant n\leqslant N} L^{-\kappa_{\rms}
                    n}\sup_{x\in\Lambda^n} | \tOne^{(N)}_{\! \!-n}(x) |^{1/2}\,.
        \end{aligned}
    \end{equation*}
    In Lemma~\ref{lem_white_noise}, we already proved $\EE \big[
        \| \xi \|_{\CB C^{- 2+\alpha - \kappa_{\rms}} (
        \TT^{d})}^{ p}
        \big]
	\lesssim p^{p/2} $, for all $ p \in (1, \infty) $.
    Therefore, it is only left to control $\tOne$, and show that 
    (while writing $ \kappa= 2\kappa_{\rms}$)
    \begin{equs}\label{eq:Kolmo}
        \begin{aligned}
            \EE\Big[ & \Big(\sup_{N\in\NN}\max_{0\leqslant n\leqslant N}L^{-\kappa
                    n}\sup_{x\in\Lambda^n}| \tOne^{(N)}_{\! \!-n}(x)
    |\Big)^p\Big]\lesssim p^{p}\,.
        \end{aligned}
    \end{equs}
 In the following, it will be convenient to use the shorthand $ \| \tOne_{\!\!
            -n}^{(N)}\|_{\infty} \eqdef \sup_{x \in \Lambda^{n}}
        | \tOne_{\!\! -n}^{(N)}(x)|$.
	
    First, recalling that $ \tOne_{\!\! -n}^{(K)} =0$ whenever $ K <0$ or $ n \geqslant
    K$, we have the telescope representation
    \begin{equs}
        \tOne^{(N)}_{\! \!-n} =\sum_{K=n}^N  \big( \tOne_{\! \!-n}^{(K)}-
        \tOne_{\! \!-n}^{(K-1)}\big)\,.
    \end{equs}
       We then crudely estimate the maximum over $n$ with a sum, which yields the upper bound
    \begin{equs}\label{eq_supp1_dbdiff}
        \max_{0\leqslant n \leqslant N}   L^{-\kappa n} \|\tOne^{(N)}_{\! \!-n} \|_\infty&\leqslant \sum_{n=0}^NL^{-\kappa n}
        \sum_{K=n}^N \| \tOne_{\! \!-n}^{(K)}-
        \tOne_{\! \!-n}^{(K-1)}\|_\infty\leqslant\sum_{K=0}^\infty\sum_{n=0}^K L^{-\kappa n}
        \| \tOne_{\! \!-n}^{(K)}-
        \tOne_{\! \!-n}^{(K-1)}\|_\infty\,,\qquad
    \end{equs}
    uniformly in $N$.
    Next, by inserting \eqref{eq_supp1_dbdiff} into \eqref{eq:Kolmo},
    we see that
    \begin{equation}\label{eq_supp_stochest_first}
        \begin{aligned}
		\Big\|\sup_{N\in\NN}\max_{0\leqslant n\leqslant N} L^{- \kappa n}\|\tOne^{(N)}_{\! \!-n} \|_\infty\Big\|_{L^p(\Omega)} & \leqslant
            \sum_{K=0}^\infty\sum_{n=0}^KL^{-\kappa n}\Big\|\big\|\tOne_{\! \!-n}^{(K)}-
            \tOne_{\! \!-n}^{(K-1)}\big\|_\infty\Big\|_{L^p(\Omega)}                                                                                                                                                             \\
                                                                                                                    & =    \sum_{K=0}^\infty\sum_{n=0}^KL^{-\kappa n}\Big\| \sup_{x \in \Lambda^n}\big| \tOne_{\! \!-n}^{(K)}-
            \tOne_{\! \!-n}^{(K-1)}(x)\big|^p \Big\|^{1/p}_{L^1(\Omega)}
            \\
                                                                                                                    & \leqslant    \sum_{K=0}^\infty\sum_{n=0}^KL^{-\kappa
            n}\Big\|\sum_{x\in\Lambda^n}\big|\tOne_{\! \!-n}^{(K)}-
            \tOne_{\! \!-n}^{(K-1)}\big|^p(x)\Big\|^{1/p}_{L^1(\Omega)}                                                                                                                                                          \\
                                                                                                                    & =   \sum_{K=0}^\infty\sum_{n=0}^KL^{-\kappa n}\Big(\sum_{x\in\Lambda^n}\EE\big[\big|\tOne_{\! \!-n}^{(K)}-
            \tOne_{\! \!-n}^{(K-1)}\big|^p(x)\big]\Big)^{1/p}                                                                                                                                                                    \\
                                                                                                                    & =
            \sum_{K=0}^\infty\sum_{n=0}^KL^{(d/p-\kappa )n}\EE\big[\big|\tOne_{\! \!-n}^{(K)}-
                \tOne_{\! \!-n}^{(K-1)}\big|^p(0)\big]^{1/p}\,,
        \end{aligned}
    \end{equation}
    where we first applied the triangle inequality, and then estimated $
        \sup_{x \in \Lambda^{n}}$ by the sum over $ \Lambda^{n} $, lastly, we used
    spatial invariance in law.
    Then, by means of Gaussian hypercontractivity, we arrive at
    \begin{equation}\label{eq:Kolmoinproof}
        \begin{aligned}
            \Big\|\sup_{N\in\NN}\max_{0\leqslant n\leqslant N} L^{- \kappa n}\|\tOne^{(N)}_{\! \!-n} \|_\infty\Big\|_{L^p(\Omega)}
             & \lesssim  p \sum_{K=0}^\infty\sum_{n=0}^KL^{(d/p-\kappa )n}\EE\big[\big(\tOne_{\! \!-n}^{(K)}-
                \tOne_{\! \!-n}^{(K-1)}\big)^2(0)\big]^{1/2}\,.
        \end{aligned}
    \end{equation}
    Thus, to conclude the proof, we only need to perform a covariance computation.
    We start by noticing that by \eqref{eq:reccov2},
    \begin{equs}
        \tOne_{\! \!-n}^{(K)}
        -
        \tOne_{\! \!-n}^{(K-1)}
        =L^{(2-2\alpha)(K-n)}
        S_{K-n}Q_{K-n}
        \big( ( P_{1} \xi_{-K})^{\diamond 2}\big)
        \,,
    \end{equs}
    where we furthermore used
    the fact that $Q_{K-n}
        \big( Q_{1} \xi_{-K}\diamond P_{1} \xi_{-K}\big)=0$, which allowed us to replace $
        \xi_{-K}\diamond P_{1} \xi_{-K}$ with $ ( P_{1} \xi_{-K})^{\diamond 2}$ in the
    expression above.
    The scaling by $ S_{K-n}$ does not affect the variance
    computation, since it only rescales the reference lattice.
    Hence,
    \begin{equation*}
        \begin{aligned}
            \EE\big[\big(\tOne_{\! \!-n}^{(K)}-
                \tOne_{\! \!-n}^{(K-1)}\big)^2(0)\big]
             & =
            L^{(4-4\alpha-2d)(K-n)} \sum_{x,y\in \Lambda^{K-n}}
            \EE[( P_{1}
                \xi_{-K})^{\diamond 2}(x)( P_{1} \xi_{-K})^{\diamond 2}(y)]\,.
        \end{aligned}
    \end{equation*}
    Therefore, using Corollary~\ref{cor_cov_P1db}, we compute
    \begin{equs}
        \EE\big[\big(\tOne_{\! \!-n}^{(K)}-
            \tOne_{\! \!-n}^{(K-1)}\big)^2(0)\big]
        &=2L^{(4-4\alpha-2d)(K-n)} \sum_{x\in\Lambda^{K-n}} \sum_{y\in\bigsquare_1(x)}
        ( \delta_{x,y} - L^{-d})^{2}
        \\
        &=2L^{(4-4\alpha-2d)(K-n)} \sum_{x\in\Lambda^{K-n}} \big(1-2L^{-d}+L^{-d}\big)\,,
    \end{equs}
    and using $d=4-2\alpha$ and $
        \kappa< \alpha$, we obtain the upper bound
	\begin{equs}\label{eq_supp2_db_diff}
        \EE\big[\big(\tOne_{\! \!-n}^{(K)}-
            \tOne_{\! \!-n}^{(K-1)}\big)^2(0)\big] &=2(1-L^{-d}) L^{-2 \alpha(K-n)}
        \leqslant 2L^{-\kappa(K-n)}\,.
    \end{equs}
    Finally, inserting \eqref{eq_supp2_db_diff} into \eqref{eq:Kolmoinproof} yields
    \begin{equation*}
        \begin{aligned}
            \Big\|\sup_{N\in\NN}\max_{0\leqslant n\leqslant N} L^{- \kappa n}\|\tOne^{(N)}_{\! \!-n} \|_\infty\Big\|_{L^p(\Omega)}
             & \lesssim p \sum_{K=0}^\infty L^{- \kappa/2 K} \sum_{n=0}^K L^{(d/p-\kappa/2
                    )n}\,.
        \end{aligned}
    \end{equation*}
    The right-hand side is summable, provided $p>2d/\kappa$, which concludes
    \eqref{eq:Kolmo}.
    The extension to all $p \in (1, \infty) $ is then a consequence
    of Markov's inequality.
\end{proof}

Having the stochastic bound from Lemma~\ref{lem_stoch} at hand, we can prove
the desired tail estimate.

\begin{proof}[Proof of Lemma~\ref{lem_tailprob}]
    Let $ g \in (0,1]$.
    We apply Chebychev's exponential inequality, which yields for every $ c >0 $
    \begin{equation*}
        \begin{aligned}
            \PP \big( \Omega_{g}^{c} \big)
            \leqslant e^{- 4 c g^{-2}} \EE\Big[
                e^{ c \Norm{ \xi, \tOne}{}^{2}}
                \Big]\,.
        \end{aligned}
    \end{equation*}
    Next, using Fubini's theorem and Lemma~\ref{lem_stoch} (recall that the
    implicit constant is independent of $p$)
    \begin{equation*}
        \begin{aligned}
            \EE\Big[ e^{ c \Norm{ \xi, \tOne}{}^{2}} \Big]
            \lesssim
            \sum_{p = 0}^{\infty} \frac{  p^{p}}{p !} (2c)^{p}
            \lesssim
            \sum_{p = 0}^{\infty} (2c\cdot e)^{p}\,,
        \end{aligned}
    \end{equation*}
    where we used Stirling's approximation in the last inequality.
    The series converges, provided we choose $ c $ small enough.
\end{proof}

\subsection{Fixed point problem}

In this section, we will prove bounds for the remainder term $
    \Psi^{(N)}_{-n}$, uniformly in $n$ and $N$.
We fix a renormalised coupling constant $g\in(0,1]$, alongside a renormalised
mass $r\geqslant1$. The estimates on the remainder $\Psi$ will hold
conditionally on $\Omega_g$.



\begin{lemma}\label{lem_fp}
    Let $d=2$.
    Then for every $r\geqslant1$ and large enough, odd $L \in \NN $ (depending only on $r$),
    we have that  for every $g\in(0,1]$
    \begin{equs}\label{eq:estimate_remainder}
        \max_{0 \leqslant n \leqslant N}
        L^{- \kappa_{\rms} n}
        \sup_{x \in \Lambda^{n}}
        | \Psi_{-n}^{(N)} (x)|^{1/3}
        \leqslant (2g)^{-1}  \,,
        \quad \text{ on }\ \Omega_g\,,
    \end{equs}
    where $ \kappa_{\rms}>0$ is the same as in the definition of $ \nnorm{\bigcdot}$ in
    \eqref{eq_stoch_norm}.
\end{lemma}

\begin{proof}
    As a preliminary step, note that we have control of the effective mass. Indeed, recalling the expression
    of the effective mass in \eqref{eq_pert_para} and using $\lambda_{-n} = L^{- \alpha
    n} g $ with $g\leqslant1$, we obtain
    \begin{equation}\label{e_mu_est}
        \begin{aligned}
            | \mu_{-n}| \leqslant  L^{(- 2+\kappa_\rms) n } ( r +c )\,,
        \end{aligned}
    \end{equation}
    for $ c\eqdef \sup_{ n \in \NN} n
        2^{-\kappa_\rms n}$ which is independent of $n$, $N$ and $L$.
	 Again, it will be convenient to abbreviate all supremum
    norms by simply writing
    \begin{equation*}
        \begin{aligned}
            \| f\|_{\infty } \eqdef \sup_{x \in \Lambda^{k}} |f(x) | \qquad \forall k \in
            \NN\,, \  f \in
            \mC ( \Lambda^{k})\,,
        \end{aligned}
    \end{equation*}
    leaving the dependency of the corresponding lattice implicit.
    Therefore, on $\Omega_g$ \eqref{eq_def_Omega}, we observe that 
    \begin{equs}\label{eq:ongoodevent}
        \lambda_{-n} \|\xi_{-n} \|_{\infty}\leqslant \frac{1}{2}  L^{(-\alpha+\kappa_\rms) n}\quad\text{and}\quad
        \lambda_{-n}^{2} \|\tOne_{\!\! -n} \|_{\infty} \leqslant\frac{1}{4}  L^{2(-\alpha+\kappa_\rms) n}\,.
    \end{equs}

    We prove the uniform bound \eqref{eq:estimate_remainder} by induction. For this, 
    we introduce
	    \begin{equs}\label{eq_triple_n}
	    \nnorm{ \xi,\tOne}_{n} \eqdef  \sup_{k \geqslant n } L^{-\kappa_{\rms}
            k}\Big(\sup_{x\in\Lambda^k} | \xi_{-k}(x) |\vee\sup_{x\in\Lambda^k} | \tOne_{\!
        \!-k}(x) |^{1/2}\Big)\,,
	\end{equs}
such that in particular $ \nnorm{\bigcdot}_{0} \equiv \nnorm{\bigcdot}$.
    Now, let $N \in \NN$ be arbitrary and note that $ \Psi_{-N}^{(N)}
        \equiv 0 $. 
We will construct the trajectory $\Psi_{-n}^{(N)}$ for $n \leqslant N$ by induction, to which
end we assume that for some $ 0 \leqslant n <N$ it holds that
\begin{equation}\label{e_ind_basis}
        \begin{aligned}
            \max_{n \leqslant k \leqslant N}
            L^{- 3\kappa_{\rms}  k }
            \sup_{x \in \Lambda^{k}}
            | \Psi_{-k}^{(N)} (x)|
             & \leqslant
     (2g)^{-1}  {\nnorm{ \xi^{(N)}, \tOne^{(N)}}{}^2_{n+1}}\,.
        \end{aligned}
    \end{equation}
Note that our induction hypothesis is stronger than \eqref{eq:estimate_remainder} 
(recall that on the event $\Omega_g$, ${\nnorm{ \xi^{(N)}, \tOne^{(N)}}{}^2_{n+1}}$ is
indeed bounded by $(2g)^{-2}$).
This will be convenient when proving convergence of $\Psi^{(N)}$ as the cutoff
$N$ is removed in Section~\ref{sec_conv}. 
In the remainder of the proof we will drop the superscripts that indicate the
    dependency on $N$. 



Before proceeding, we recall the discussion at the bottom of Remark~\ref{rem:3_1} about the triangular structure of the system~\eqref{eq:sys}: the remainder $ \Psi_{-n+1}$ in \eqref{e_rg_system}
is not determined solely by $ \Psi_{-n}$ but also depends on the fluctuation field $
\mR_{-n +1} $.
Accordingly, each induction step is divided into two parts: we first estimate $ \mR_{-n
+1}$ via
a fixed--point problem, and then derive the desired bound \eqref{e_ind_basis} for $ \Psi_{-n +1}$.

    \begin{enumerate}
	    \item \textit{Remainder fluctuation field.}
		    Using \eqref{e_supp1_rho} together with the fact that $ | P_{1}f | \leqslant 2 \| f\|_{ \infty}$,
              we have
              \begin{equation}\label{eq_supp1_rho_est}
                  \begin{aligned}
                      \|\mR_{-n+1}\|_{\infty}
                       & \leqslant
                      2 \big(
                      \lambda_{-n} \|\xi_{-n} \|_{\infty}
                      +
                      \lambda_{-n}^{2} \|\tOne_{\!\! -n} \|_{\infty}
                      +
                      \lambda_{-n}^{3} \|\Psi_{-n}\|_{\infty}
                      +
                      \mu_{-n}
                      \big) \|\mR_{-n+1}\|_{\infty}                                        \\
                       & \quad +
                      4\Big(
                      \|\tOne_{\!\! -n}\|_{\infty}
                      +  \| \xi_{-n}\|^{2}
                      +\lambda_{-n} \|\Psi_{-n} \|_{\infty}
                      +
                      \lambda_{-n} \| \tOne_{\!\! -n} \|_{\infty}   \|\xi_{-n} \|_{\infty} \\
                       & \quad \quad\quad  +
                      \lambda_{-n}^{ 2} \| \Psi_{-n} \|_{\infty} \| \xi_{-n} \|_{ \infty }
                      +
                      \mu_{-n} \lambda_{-n}^{-1} \| \xi_{-n}\|_{\infty}\Big)
                      \,.
                  \end{aligned}
              \end{equation}

              Now, using the estimates \eqref{e_mu_est}, \eqref{eq:ongoodevent}, and \eqref{e_ind_basis}, we
	      can upper bound the first term on the right-hand side of \eqref{eq_supp1_rho_est}:
              \begin{equs}
                  \begin{aligned}
                      2\big(
                      \lambda_{-n} & \|\xi_{-n} \|_{\infty}
                      +
                      \lambda_{-n}^{2} \|\tOne_{\!\! -n} \|_{\infty}
                      +
                      \lambda_{-n}^{3} \|\Psi_{-n}\|_{\infty}
                      +
                      \mu_{-n}
                      \big)\|\mR_{-n+1}\|_{\infty}          \\
                                   & \leqslant
                      2\big(
                      \tfrac{1}{2} L^{(-\alpha+\kappa_\rms) n} +
                      \tfrac{1}{4}  L^{2(-\alpha+\kappa_\rms) n}
                      +
                      \tfrac{1}{8} L^{ 3(-\alpha+\kappa_\rms) n}
                      +
                      L^{(- 2 +\kappa_\rms )n } ( r +c )
                      \big)\|\mR_{-n+1}\|_{\infty}          \\
                                   & \leqslant
                      (1+3L^{-1/2}C_r)
                      L^{(- \alpha+\kappa_\rms) n}
                      \|\mR_{-n+1}\|_{\infty}\leqslant
                      (1+3L^{-1/2}C_r)
                      L^{-1/2}
                      \|\mR_{-n+1}\|_{\infty}
                      \,,
                  \end{aligned}
              \end{equs}
              where we introduced the notation  $ C_r\eqdef  (1 + r+c) >1$, used the induction hypothesis for $ \Psi_{-n}$ in the first inequality, and crude estimates such as $(-\alpha+\kappa_\rms)n\leqslant -1/2$.

              Along the same lines, we can control the second terms on the right--hand side of
              \eqref{eq_supp1_rho_est}:
              \begin{equs}
                    4\Big(
                      &\|\tOne_{\!\! -n}\|_{\infty}
                      +  \| \xi_{-n}\|^{2}
                      +\lambda_{-n} \|\Psi_{-n} \|_{\infty}
                      +
		      \big(
                      \lambda_{-n} \| \tOne_{\!\! -n} \|_{\infty}   +
                      \lambda_{-n}^{ 2} \| \Psi_{-n} \|_{\infty}                       +
                      \mu_{-n} \lambda_{-n}^{-1} \big) \|\xi_{-n}\|_{\infty}\Big)\\&\leqslant
                  2\Big(  2 L^{2\kappa_\rms n} +2L^{(-\alpha+3\kappa_\rms)
		  n}+L^{(-2\alpha+4\kappa_\rms) n}+L^{(-2+\alpha+2\kappa_\rms)
  n}(r+c)\Big)g^{-1}{\nnorm{\xi^{(N)}, \tOne^{(N)}}_{n}}\\
                  &\leqslant 4
                      (1 +4 L^{-1/2}C_r)
                      L^{2\kappa_\rms n}g^{-1}
		     {\nnorm{\xi^{(N)}, \tOne^{(N)}}_{n}}\,,
                                \end{equs}
              and finally arrive at
	      \begin{equation}\label{eq_supp123_R}
                  \begin{aligned}
                      \| \mR_{-n +1}\|_{\infty}
                       & \leqslant
                      (1+3L^{-1/2}C_r )
                      L^{-1/2}
                      \|\mR_{-n+1}\|_{\infty}
                      +
                      4
                      (1 +4 L^{-1/2}C_r)
                      L^{2\kappa_\rms n}g^{-1}
		     {\nnorm{\xi^{(N)}, \tOne^{(N)}}_{n}} \\
                       & \leqslant
                      \tfrac{1}{2}\|\mR_{-n+1}\|_{\infty}
                      +
                      8
                      L^{2\kappa_\rms n}g^{-1}
	     {\nnorm{\xi^{(N)}, \tOne^{(N)}}_{n}} \\
                       & \leqslant16
                      L^{2 \kappa_{\rms} n}g^{-1}
		      {\nnorm{\xi^{(N)}, \tOne^{(N)}}_{n}}\,,
                  \end{aligned}
              \end{equation}
              where again we used the definition of $ C_{r}$ and then took $L$ large enough depending on $r$. 
	      In particular, $\CB R_{-n+1}$ only depends on the randomness at scales lower or equal to $L^{-n}$.

        \item \textit{Remainder effective force.} We perform the same type of estimate for $ \Psi_{-n+1}$: From
              \eqref{e_rg_system}, we see that
              \begin{equs}\label{eq_supp1_psi_est}
                  \begin{aligned}
                      \|\Psi_{-n+1}\|_{\infty}
                       & \leqslant
                      L^{2- 3 \alpha} \big(
                      1+
                      2\lambda_{-n} \| \xi_{-n}\|_{\infty}+
                      \lambda_{-n}^{2} \|\mR_{-n+1}\|_{\infty}
                      \big)
                      \| \Psi_{-n}\|_{\infty}
                      \\
                       & \quad+
                      L^{2- 3 \alpha} \big(
                      2 \|\tOne_{\!\! -n}\|_{\infty} \| \xi_{-n}\|_{\infty}
                      +
                      \| \xi_{-n} \|_{ \infty} \|\mR_{-n+1}\|_{\infty}
                      +
                      \lambda_{-n}\|\tOne_{\!\! -n}\|_{\infty} \|\mR_{-n+1}\|_{\infty}
                      \big)
                      \,,\textcolor{white}{blablabl}
                  \end{aligned}
              \end{equs}
              where we again used $ \| P_{1}f \|_{\infty} \leqslant 2 \| f \|_{\infty}$.
              Inserting  \eqref{eq:ongoodevent}, \eqref{e_ind_basis} and \eqref{eq_supp123_R},
              we deduce the upper bounds
              \begin{equs}\label{eq_supp2_psi_est}
                  {}\big(
                  1+
                  2\lambda_{-n} \| \xi_{-n}\|_{\infty}+
                  \lambda_{-n}^{2} \|\mR_{-n+1}\|_{\infty}
                  \big)
                  \| \Psi_{-n}\|_{\infty}
                  \leqslant
                  L^{ 3\kappa_{\rms} n}
	  g^{-1} {\nnorm{ \xi^{(N)}, \tOne^{(N)}}{}^2_{n}}
              \end{equs}
	      and
	      \begin{equation}\label{eq_supp3_psi_est}
                  \begin{aligned}
                      \big(
                      2 \|\tOne_{\!\! -n}\|_{\infty} \| \xi_{-n}\|_{\infty}
                      +
                      \| \xi_{-n} \|_{ \infty} \|\mR_{-n+1}\|_{\infty}
                      + &
                      \lambda_{-n}\|\tOne_{\!\! -n}\|_{\infty} \|\mR_{-n+1}\|_{\infty}
                      \big)         \\
                        & \leqslant
			10
                      L^{3 \kappa_{\rms} n}
		      g^{-1} {\nnorm{ \xi^{(N)}, \tOne^{(N)}}{}^2_{n}}\,.
                  \end{aligned}
              \end{equation}
              Inserting \eqref{eq_supp2_psi_est} and \eqref{eq_supp3_psi_est} into \eqref{eq_supp1_psi_est} then yields
              \begin{equation}\label{e_}
                  \begin{aligned}
                      \|\Psi_{-n+1}\|_{\infty}
                      \leqslant
		      \big( 11 L^{2-3 \alpha}   \big) L^{3 \kappa_{\rms} n} g^{-1}
		      {\nnorm{ \xi^{(N)}, \tOne^{(N)}}{}^2_{n}}\,.
                  \end{aligned}
              \end{equation}
              Once more, we may choose $ L$
	      large enough, such that $ 11 L^{2-3 \alpha}   \leqslant 2^{-1}$.
    \end{enumerate}
    Overall, this extends \eqref{e_ind_basis} to $n-1 \leqslant k \leqslant N$ and
    concludes the induction step.
\end{proof}

\subsection{Control of the renormalised solution}\label{sec_besov}

In the previous section, we proved a uniform bound (in $n$ and $N$) of the remainder term $
    \Psi_{-n}^{(N)}$ in the effective equation~\eqref{e_v_n}. This provides the
last ingredient to control the coefficients in the effective equation on the event $\Omega_g$, where the enhanced
noise $ ( \xi^{(N)}, \tOne^{(N)})$ remains controlled by $ g$.
It is only left to push this control from the coefficients onto the solution $ v_{-n}$, 
and derive that
\begin{equation*}
    \begin{aligned}
        \sup_{N \in \NN}\Big(|v_0^{(N)}|\vee \max_{ 1 \leqslant n  \leqslant N}
        L^{ - \kappa n}
        \sup_{ x \in \Lambda_{n}} \big| P_{1} v^{(N)}_{-n} (x ) |\Big) < \infty\,.
    \end{aligned}
\end{equation*}
The key ingredient of this control
 is an inductive scheme backwards along the flow, which we
present next. This is in contrast to the previous section, where we propagated
bounds forward along the flow.
\begin{lemma}\label{lem_aid_besov}
	Let  $d=2$.
    Then for every $r\geqslant1$ and large enough, odd $L \in \NN $ (depending only on $r$),
    we have that  for every $g\in(0,1]$ it holds that
    \begin{equs}\label{eq:Q_1andP_1}
        \begin{aligned}
            \sup_{N \in \NN}\Big(|v_0^{(N)}|\vee \max_{1\leqslant n \leqslant N}
            \sup_{x \in \Lambda^{n}} \Big(   L^{-( \alpha+\kappa_\rms)n}| Q_{1} v^{(N)}_{-n}(x) |
            \vee
            L^{-3\kappa_{\rms}n}  |P_{1} v^{(N)}_{-n}(x)|
            \Big)\Big) \leqslant 8\,, \quad \text{ on } \Omega_g\,. \qquad
        \end{aligned}
    \end{equs}
\end{lemma}

Before turning to the proof, let us briefly explain the strategy and the origin of the estimate~\eqref{eq:Q_1andP_1}.
Recall that $v_0$ is simply a scalar, determined by the algebraic equation~\eqref{e_v_n} for $n=0$, where the Laplacian term vanishes.
It is therefore the first object to be defined.
Our goal is then to construct the family $(v_{-n})_{n\geq 0}$ recursively by induction
over~$n$, passing from $v_{-n}$ to $v_{-n-1}$.
Assuming that $v_{-n}$ is already known,
we leverage on the fact that $Q_1 v_{-n-1}$ is obtained from $v_{-n}$ by a simple rescaling.
Consequently, it only remains to control $P_1 v_{-n-1}$, 
where we will use~\eqref{eq_supp2_z} and the bound on $Q_1 v_{-n-1}$.

The estimate~\eqref{eq:Q_1andP_1} involves a minor subtlety, arising from the fact that
we are solving an equation whose solution has positive regularity.
While $P_1 v_{-n}$ grows at most like $L^{\kappa n}$, the term $Q_1 v_{-n}$ exhibits a worse behaviour.
This phenomenon is analogous to a classical observation: If $f$ has H\"older--Besov regularity $\beta>0$ and $(\Delta_i)_{i\geq 0}$ denotes a Littlewood--Paley decomposition, then $\|\Delta_i f\|_\infty \sim 2^{-\beta i}$ whereas $\big\|\sum_{k=0}^i \Delta_k f\big\|_\infty \sim 1$.
Since we expect the solution on $ \TT^d$ to have regularity $\alpha-\kappa$, rescaling to
the unit lattice suggests that $Q_1 v_{-n}$ should grow like $L^{(\alpha+\kappa)n}$.

\begin{proof}[Proof of Lemma~\ref{lem_aid_besov}]
    Let $N \in \NN$.
    We recall the notation from Section~\ref{sec_rg_general} and write
    \begin{equation*}
        \begin{aligned}
            w_{-n +1} =  Q_{1} v_{-n}\quad \text{and}\quad z_{-n +1}=
            P_{ 1} v_{-n}\,.
        \end{aligned}
    \end{equation*}
    The idea of the proof is built on the fact that
    $v_{-n} = L^{\alpha} S_{-1} v_{-n+1} + z_{-n +1}$.
    In order to control $ v_{-n}$, it suffices to control $
        v_{-n +1}$ and the fluctuations that were averaged out going from the lattice $
        \Lambda^{n}$ to the lattice $ \Lambda^{n-1}$.

    First, to control the fluctuations, we make use of \eqref{eq_supp2_z} which yields
    \begin{equation*}
        \begin{aligned}
            \|z_{-n +1}\|_{\infty}
             & \leqslant
            2\big(\lambda_{-n} \|  \xi_{-n}\|_{\infty} + \lambda_{-n}^{2} \| \tOne_{\!\!
                -n}\|_{\infty}
            +\lambda_{-n}^{3} \|   \Psi_{-n} \|_{\infty} \big) \|w_{-n+1}\|_{\infty} \\
             & \qquad+
            2\big(\lambda_{-n} \| \xi_{-n}\|_{\infty} +
            \lambda_{-n}^{2} \| \tOne_{\!\! -n} \|_{\infty}
            +
            \lambda_{-n}^{3} \|   \Psi_{-n} \|_{\infty}
            + \mu_{-n} \big) \| z_{-n+1}\|_{\infty}
        \end{aligned}
    \end{equation*}
    where we again used $ \| P_{1} f \|_{\infty} \leqslant 2 \| f\|_{\infty}$, as
    well as the notation $ \| \bigcdot \|_{\infty}$ for the
    supremum--norm without specifying the underlying lattice.
    Once more inserting that $ \lambda_{-n} =L^{-\alpha n}g $, \eqref{e_mu_est},
    \eqref{eq:ongoodevent} and Lemma~\ref{lem_fp}, we arrive at
    \begin{equs}\label{eq_bound_zWithQ}
        \begin{aligned}
             & \|z_{-n +1}\|_{\infty}                                                                                     \\
             & \leqslant
            \big( L^{ (-  \alpha + \kappa_{\rms}) n}  +  L^{  2(-  \alpha + \kappa_{\rms}) n}
            +   L^{ 3 ( -\alpha+ \kappa_{\rms})n} \big)
            \big(\|w_{-n+1}\|_{\infty} + \| z_{-n +1}\|_{ \infty}\big)
            + 2L^{(-2+\kappa_\rms)n}( r + c ) \| z_{-n+1}\|_{\infty}                                                      \\
             & \leqslant \big(3 L^{-1/2} +2L^{-1}(r+c)\big)  \| z_{-n +1} \|_{\infty}+2L^{ (-  \alpha + \kappa_{\rms}) n}
            \| Q_{1} v_{-n}\|_{\infty}
            \\
             & \leqslant \tfrac{1}{2}   \| z_{-n +1} \|_{\infty}+2L^{ (-  \alpha + \kappa_{\rms}) n}
            \| Q_{1} v_{-n}\|_{\infty}                                                                                    \\
             & \leqslant4L^{ (-  \alpha + \kappa_{\rms}) n}
            \| Q_{1} v_{-n}\|_{\infty}  \\
             & \leqslant(4L^{ -  \alpha + \kappa_{\rms}})^n
            \| Q_{1} v_{-n}\|_{\infty} \,,
        \end{aligned}
    \end{equs}
    for every $ n \geqslant 1$, where again we took $L$ large enough depending on $r$.

    Let us now proceed by induction over $n$. First, we observe that in view of the
    definition \eqref{eq_hierLaplace_BBS}, $\Delta_{H,0}=0$. Therefore, equation
    \eqref{e_v_n} becomes (for $n=0$)
    \begin{equs}
        \big(g\xi_{0}
        + g^{2} \tOne_{\! \! 0}
        + g^{3} \Psi_{0} + r\big) v_{0}  + 1=0\,,
    \end{equs}
    which is an algebraic equation in $ v_{0} \in \RR$. 
    Here we used that $ \lambda_{0} = g $, $ \mu_{0}= r$, and $
        \gamma_{0} =1 $. Therefore,
    \begin{equation*}
        \begin{aligned}
            | v_{0}| & =
            \big|
            g \xi_{0} + g^{2} \tOne_{0}+ g^{3}
            \Psi_{ 0} + r
            \big|^{-1} \leqslant (  r -g| \xi_{0}| - g^{2} |\tOne_{0}|- g^{3}
            |\Psi_{ 0}| )^{-1}
            \leqslant \big(  r -\tfrac{1}{2}  - \tfrac{1}{4} - \tfrac{1}{8}  \big)^{-1}\leqslant8
            \,,
        \end{aligned}
    \end{equation*}
    where we used \eqref{eq:ongoodevent} and Lemma~\ref{lem_fp} in the second inequality, alongside the fact that $r\geqslant1$.
    Hence, by definition of $ v_{-1}$ and \eqref{eq_bound_zWithQ},
    we therefore conclude that
    \begin{equation*}
        \begin{aligned}
            \| Q_{1} v_{-1}\|_{\infty} \leqslant L^{\alpha} | v_{0}|
            \leqslant 8 L^{\alpha}
            \qquad \text{and thus}\qquad
            \| z_{0}\|_{\infty} \leqslant
            32 L^{ \kappa_{\rms}} \,.
        \end{aligned}
    \end{equation*}

    We proceed by induction and assume that
    \begin{equation}
        \begin{aligned}
            \| Q_{1} v_{-n}\|_{\infty}
            \leqslant (8 L^{\alpha })^{n}
            \qquad \text{and thus}\qquad
            \| z_{-n+1}\|_{\infty} \leqslant
            (32L^{ \kappa_{\rms}})^{n} \,,
        \end{aligned}
    \end{equation}
    for some $1 \leqslant n < N$. Then,
    \begin{equation}\label{eq_supp_istep}
        \begin{aligned}
            \| Q_{1} v_{-n-1}\|_{\infty}
             & \leqslant L^{\alpha} \| v_{-n}  \|_{\infty}
            \leqslant L^{\alpha} \big(
            \| Q_{1} v_{-n}\|_{\infty}
            + \| z_{-n+1}\|_{\infty} \big)                                                     \\
             & \leqslant  \tfrac{1}{8} (8L^{\alpha})^{ n+1}+  L^{\alpha} (32 L^{\kappa_{\rms}}
            )^{n} \leqslant
            ( 8 L^{\alpha})^{n+1} \big(
            \tfrac{1}{8} +\tfrac{1}{2}L^{ - \alpha + \kappa_{\rms}}
            \big) \leqslant
            ( 8 L^{\alpha})^{n+1} \,,
        \end{aligned}
    \end{equation}
    where we used the induction hypothesis in the second line, and in the last
    inequality that $L$ is large enough (independent of $N$).

    To conclude the induction step, it only remains to control the fluctuation term
    $ \| z_{-n}\|_{\infty}$, for which we use \eqref{eq_bound_zWithQ} and
    \eqref{eq_supp_istep}, which yields
    \begin{equation*}
        \begin{aligned}
            \| P_{1} v_{-n-1}\|_{\infty}
            =
            \| z_{-n}\|_{\infty} \leqslant
           ( 4 L^{ -  \alpha+ \kappa_{\rms} })^{n+1}
            \| Q_{1} v_{-n-1}\|_{\infty}
            \leqslant
            (32 L^{\kappa_{\rms}})^{n+1}\,.
        \end{aligned}
    \end{equation*}
    Finally, taking $L$ large enough, we can enforce that $8\leqslant L^{\kappa_\rms}$. This concludes the proof.
\end{proof}

\section{Convergence in the UV limit}\label{sec_conv}

In the previous section, we proved control of the RG flow of the enhanced noise, the remainder
$ \Psi $, and the effective solution.
Following, the same steps as in the proofs of Lemmas~\ref{lem_stoch}, \ref{lem_fp}, and
\ref{lem_aid_besov}, we can show that the sequences are in fact Cauchy, and thus converge.
This extension is standard for readers familiar with subcritical SPDEs and may be skipped.
However, we provide the complete argument in this section for any reader that
is less familiar with the topic.

\subsection{Convergence of the effective force}

First, we show that the enhanced noise converges $ \PP$--almost surely in the UV limit.

\begin{lemma}[Convergence of the enhanced noise]\label{lem_stoch_conv}
    Let $\alpha>0$, i.e. $d<4$, and $ L \in \NN$.
    Then for every $g \in (0, \infty)$ the following
    limit exists $ \PP$--almost surely
    \begin{equs}
	    \lim_{N \to \infty} \big(\xi^{(N)},\tOne^{(N)}
	    \big) =( \xi_{-n}, \tOne^\star_{-n})_{n \in \NN}\,,
    \end{equs}
    in the topology induced by $\nnorm{\bigcdot}$.
\end{lemma}

\begin{proof}
	It is not difficult to check that $ ( \xi^{(N)})_{N \in \NN}$ converges almost surely to $ (\xi_{-n})_{n \in \NN} $ in
	$\nnorm{\bigcdot} $.
	To conclude convergence of the enhanced noise, it is therefore only left to show that 
$ (\tOne^{(N)})_{N \in \NN}$ forms a Cauchy sequence. More precisely, we prove that
$ \PP$ --almost surely
\begin{equation}\label{eq_cauchy_db}
	\begin{aligned}
		\sup_{ M,N \in \NN  } L^{ \kappa/4 N}\max_{0\leqslant n\leqslant
		{N+M}}
		L^{- \kappa n} \sup_{x \in
                    \Lambda^{n}} |\tOne^{(N+M)}_{\! \!-n}-\tOne^{(N)}_{\! \!-n} |(x)<
		    \infty\,,
	\end{aligned}
\end{equation}
    where $ \kappa:= 2 \kappa_{\rms}$ for convenience.
The proof carries over almost verbatim from the proof
    of Lemma~\ref{lem_stoch}, however, we include the main steps for a
    complete account.
    
    Again, we start from the telescope representation
    \begin{equs}
        \tOne^{(N+M)}_{\! \!-n}-\tOne^{(N)}_{\! \!-n} =\sum_{K=(N\vee n)+1}^{N+M}  \big( \tOne_{\! \!-n}^{(K)}-
        \tOne_{\! \!-n}^{(K-1)}\big)\,,
    \end{equs}
    which allows us to crudely replace suprema with sums (as we did in
    \eqref{eq_supp1_dbdiff}), such that
    \begin{equs}
        \begin{aligned}
            \sup_{M,N \in \NN} L^{ \kappa/4 N}
	    \max_{0\leqslant n \leqslant {N+M}}
            L^{-\kappa n} \|\tOne^{(N+M)}_{\! \!-n}-\tOne^{(N)}_{\! \!-n} \|_{\infty}
	     & \leqslant {\sup_{M \in \NN} } \sum_{N =0}^{\infty}
            L^{ \kappa/4 N}
    {\sum_{n =0}^{N+M}}
	    L^{ - \kappa n}
	    \sum_{K=(N\vee n)+1}^{{N+M}}   \| \tOne_{\! \!-n}^{(K)}-
            \tOne_{\! \!-n}^{(K-1)}\|_{\infty}\\
	    & \leqslant \sum_{N =0}^{\infty}
            L^{ \kappa/4 N}
	    \sum_{K=N+1}^{\infty}  
		\sum_{n =0}^{K}
	    L^{ - \kappa n}
\| \tOne_{\! \!-n}^{(K)}-
            \tOne_{\! \!-n}^{(K-1)}\|_{\infty}\,.
    \end{aligned}
    \end{equs}
     Thus, in complete analogy to \eqref{eq_supp_stochest_first}, we have
    \begin{equation*}
        \begin{aligned}
             & \Big\|
            \sup_{M,N \in \NN} L^{ \kappa/4 N}
	    \max_{0\leqslant n \leqslant {N+M}}
            L^{-\kappa n} \|\tOne^{(N+M)}_{\! \!-n}-\tOne^{(N)}_{\! \!-n} \|_{\infty}
            \Big\|_{L^{p} ( \Omega)}                          \\
             & \leqslant
            \sum_{N =0}^{\infty}
            L^{ \kappa/4 N}
            {\sum_{K=N+1}^{\infty} 
	        \sum_{n =0}^{K}}
    L^{ - \kappa n} \Big\| \sup_{x \in \Lambda^{n}} | \tOne_{\! \!-n}^{(K)}-
            \tOne_{\! \!-n}^{(K-1)}|(x) \Big\|_{L^{p} ( \Omega)} \\
             & \leqslant
            \sum_{N =0}^{\infty} L^{\kappa/4 N}
            \sum_{K= N+1}^{ \infty}\sum_{n=0}^KL^{(d/p-\kappa )n}\EE\big[\big|\tOne_{\! \!-n}^{(K)}-
                \tOne_{\! \!-n}^{(K-1)}\big|^p(0)\big]^{1/p}
            \,.
        \end{aligned}
    \end{equation*}
    Applying additionally hypercontractivity in combination with \eqref{eq_supp2_db_diff}, we conclude
    \begin{equs}
        \begin{aligned}
            \Big\|
            \sup_{M,N \in \NN} L^{ \kappa/4 N}
	    \max_{0\leqslant n \leqslant {N+M}}
            L^{-\kappa n} \|\tOne^{(N+M)}_{\! \!-n}-\tOne^{(N)}_{\! \!-n} \|_{\infty}
            \Big\|_{L^{p} ( \Omega)}
             & \lesssim p
            \sum_{N =0}^{\infty} L^{\kappa/4 N}
            \sum_{K= N+1}^{ \infty}\sum_{n=0}^KL^{(d/p-\kappa )n}
            L^{ - \kappa/2 (K-n)} \\
             & \leqslant  p
            \sum_{N =0}^{\infty} L^{-\kappa/4 N}
            \sum_{K= 1}^{ \infty}
            L^{- \kappa/2 K}
            \sum_{n=0}^\infty L^{(d/p-\kappa/2 )n}
            \lesssim p\,,
        \end{aligned}
    \end{equs}
    where we performed an index shift ($K \mapsto K-N$) in the second inequality,
    and assumed that $p$ is large enough such that the most inner sum converges.
    This moment estimate concludes in particular that \eqref{eq_cauchy_db} is
    finite, $ \PP$--almost surely.
\end{proof}

Likewise, the remainder $ \Psi^{(N)}$ forms a Cauchy sequence on $ \Omega_{g} $, and
converges. 
\begin{lemma}[Convergence of the remainder]\label{lem_remainder_conv}
	 Let $d=2$.
    For every $r \geqslant 1$, there exists an odd $L \in \NN$, sufficiently large
    (and depending only on $r$), such that for every $g \in (0,1]$ 
    the limit   $ \Psi^{\star} = (\Psi_{-n}^{\star})_{n \in \NN} $ exists on $
    \Omega_{g}$, such that
    \begin{equation*}
        \begin{aligned}
            \lim_{N \to \infty} \max_{n \geqslant 0}
            L^{- \kappa_{\rms} n}
            \sup_{x \in \Lambda^{n}}
            | \Psi_{-n}^{(N)} (x)- \Psi^{\star}_{-n} (x)|^{1/3}=0\,,
            \qquad \text{on } \Omega_{g}\,,
        \end{aligned}
    \end{equation*}
    where $ \kappa_{\rms}>0$ is the same as in \eqref{eq_stoch_norm}.
    Here we set $ \Psi_{-n}^{(N)} \equiv 0$ for all $ n >N$.
\end{lemma}

\begin{proof}
	In the following, we assume to be on the good event $\Omega_{g}$.
    Again, we will prove that $ ( \Psi_{-n}^{(N)})_{n \in \NN}$ is Cauchy.
    More precisely, we show that
    \begin{equs}\label{eq_supp_rem_conv}
	    \begin{aligned}
		    &{\max_{0 \leqslant n \leqslant N+M}}
        L^{- \kappa_{\rms} \, n }
        \sup_{x \in \Lambda^{n}}
        | \Psi^{(N+M)}_{-n} - \Psi^{(N)}_{-n}|^{1/3}(x)
		    \leqslant 
	(2g)^{-1/3}
	\nnorm{ (\xi^{(N+M)},
            \tOne^{(N+M)})-(\xi^{(N)}, \tOne^{(N)})}{}^{2/3}
        \,,\qquad
	    \end{aligned}
    \end{equs}
    cf. Lemma~\ref{lem_fp}.
    Notice that by Lemma~\ref{lem_stoch_conv},
    the right--hand side vanishes on $ \Omega_{g}$, as $ N \to \infty$.\\

    We follow the same steps as in the proof of Lemma~\ref{lem_fp}, using the norm
    $\nnorm{\bigcdot}_{n}$ from \eqref{eq_triple_n} which only considers scales above $n$.
    First, we assume that the following hypothesis is true
    \begin{equation*}
        \begin{aligned}
             & \max_{n \leqslant k \leqslant N+M}
            L^{- 3\kappa_{\rms} \, k }
            \sup_{x \in \Lambda^{k}}
            | \Psi^{(N+M)}_{-k} (x) - \Psi^{(N)}_{-k} (x)|
            \leqslant
	    ( 2 g)^{-1}
             {\nnorm{ (\xi^{(N+M)},
	     \tOne^{(N+M)})-(\xi^{(N)}, \tOne^{(N)})}{}^{2}_{n+1}}\,.
        \end{aligned}
    \end{equation*}
    for some $0 \leqslant n < N $. Notice that for $N\leqslant n \leqslant  N+M$ the
    statement reduces to \eqref{e_ind_basis}, because in this case $ \Psi^{(N)}_{-n}\equiv
 0 $ and $( \xi^{(N)}_{-k}, \tOne_{-k}^{(N)})\equiv0$ for every $k\geqslant n+1$. We proceed
       inductively.

    \begin{enumerate}
        \item \textit{Remainder fluctuation field.} Using again \eqref{e_supp1_rho} together with the fact that $ | P_{1}f |
                  \leqslant 2 \| f\|_{ \infty}$, yields (after adding and subtracting multiple zero terms
              and applying the triangle inequality)
              \begin{equation*}
                  \begin{aligned}
                       & \| \mR_{-n+1}^{(N+M)}- \mR_{-n+1}^{(N)} \|_{\infty}                                             \\
                       & \leqslant
                      2 \big(
                      \lambda_{-n} \|\xi_{-n}^{(N)} \|_{\infty}
                      +
                      \lambda_{-n}^{2} \|\tOne_{\!\! -n}^{(N)} \|_{\infty}
                      +
                      \lambda_{-n}^{3} \|\Psi_{-n}^{(N)}\|_{\infty}
                      +
                      \mu_{-n}^{(N)}
                      \big) \|\mR_{-n+1}^{(N+M)}- \mR_{-n +1}^{(N)}\|_{\infty}                                           \\
                       & \quad +
                      2 \Big(
                      \lambda_{-n}^{2} \|\tOne_{\!\! -n}^{(N+M)}-\tOne_{\!\! -n}^{(N)} \|_{\infty}
                      +
                      \lambda_{-n}^{3} \|\Psi_{-n}^{(N+M)}-\Psi_{-n}^{(N)}\|_{\infty}
                      \Big) \|\mR_{-n+1}^{(N+M)}\|_{\infty}                                                              \\
                       & \quad +
                      4\Big(
                      \|\tOne_{\!\! -n}^{(N+M)}-\tOne_{\!\! -n}^{(N)}\|_{\infty}
                      +\lambda_{-n} \|\Psi_{-n}^{(N+M)}- \Psi_{-n}^{(N)} \|_{\infty}
                      +
                      \lambda_{-n} \|\tOne_{\!\! -n}^{(N+M)}- \tOne_{\!\! -n}^{(N)} \|_{\infty}   \|\xi_{-n} \|_{\infty} \\
                       & \quad \quad\quad  +
                      \lambda_{-n}^{ 2} \|\Psi_{-n}^{(N+M)}-  \Psi_{-n}^{(N)} \|_{\infty} \| \xi_{-n} \|_{ \infty }
                      \Big)
                      \,.
                  \end{aligned}
              \end{equation*}
              Because each term can be upper bounded using the definition of
	      $\Omega_{g}$,
	      \eqref{e_ind_basis}, or the induction
              hypothesis, we see that
	      \begin{equation}\label{eq_R_diff}
                  \begin{aligned}
		      &\| \mR_{-n+1}^{(N+M)}- \mR_{-n+1}^{(N)} \|_{\infty}\\
		       & \leqslant 
		       \tfrac{1}{2}\| \mR_{-n+1}^{(N+M)}- \mR_{-n+1}^{(N)} \|_{\infty} 
		       + 
		        8
                      L^{2\kappa_\rms n}g^{-1} {\nnorm{ (\xi^{(N+M)},
		      \tOne^{(N+M)}) - ( \xi^{(N)}, \tOne^{(N)} ) }_{n}}
\\
		       & \leqslant 
		       16
                      L^{2\kappa_\rms n}g^{-1} {\nnorm{ (\xi^{(N+M)},
		      \tOne^{(N+M)}) - ( \xi^{(N)}, \tOne^{(N)} ) }_{n}}
		       \,,
                  \end{aligned}
              \end{equation}
	      which follows verbatim from the steps performed in \eqref{eq_supp123_R}.
	      
        \item \textit{Remainder effective force.} Likewise, we can estimate the difference of the $\Psi$--terms using
              \eqref{e_rg_system}:
              \begin{equs}
                  \begin{aligned}
                      \|\Psi_{-n+1}^{(N+M)}-\Psi_{-n+1}^{(N)} \|_{\infty}
                       & \leqslant
                      L^{2- 3 \alpha} \big(
                      1+
                      2\lambda_{-n} \| \xi_{-n}\|_{\infty}+
                      \lambda_{-n}^{2} \|\mR_{-n+1}^{(N)}\|_{\infty}
                      \big)
                      \| \Psi_{-n}^{(N+M)}- \Psi_{-n}^{(N)}\|_{\infty}
                      \\
                       & \quad+
                      L^{2- 3 \alpha}
                      \lambda_{-n}^{2} \|\mR_{-n+1}^{(N+M)}-\mR_{-n+1}^{(N)}\|_{\infty}
                      \| \Psi_{-n}^{(N+M)} \|_{\infty}
                      \\
                       & \quad+
                      L^{2- 3 \alpha} \big(
                      2 \|\tOne_{\!\! -n}^{(N+M)}-\tOne_{\!\! -n}^{(N)}\|_{\infty} \| \xi_{-n}\|_{\infty}
                      +
                      \| \xi_{-n} \|_{ \infty} \|\mR_{-n+1}^{(M+N)}- \mR_{-n+1}^{(N)}\|_{\infty}\big) \\
                       & \quad
                      +L^{2- 3 \alpha} \big(
                      \lambda_{-n}\|\tOne_{\!\! -n}^{(N+M)} -\tOne_{\!\! -n}^{(N)}\|_{\infty}
                      \|\mR_{-n+1}^{(N)}\|_{\infty}                                                   \\
                       & \qquad\qquad+
                      \lambda_{-n}\|\tOne_{\!\! -n}^{(N+M)}\|_{\infty} \|\mR_{-n+1}^{(N+M)}-\mR_{-n+1}^{(N)}\|_{\infty}
                      \big)
                      \,.
                  \end{aligned}
              \end{equs}
              Therefore, using the induction hypothesis for increments of $
	      \Psi_{-n}^{(N)}$, the upper bound 
	      \eqref{eq_R_diff}, as well
	      as the definition of $ \Omega_{g}$ and \eqref{e_ind_basis}, we have
	      extended the induction hypothesis to $
	      \Psi_{-n+1}^{(N)}$.
    \end{enumerate}
    This concludes the proof.
\end{proof}

\subsection{Convergence of the solution and proof of Proposition~\ref{prop_main}}

Next, we conclude that also the trajectories $ (v^{(N)}_{-n})_{n \leqslant N}$ form a
	Cauchy sequence. 
	To this end, recall that $ w_{-n +1}^{(N)}=Q_{1} v_{-n}^{(N)}$ and $ z_{-n+1}^{(N)}= P_{1}
	v^{(N)}_{-n} $. Therefore, it suffices to prove
	\begin{equation}\label{eq_}
		\begin{aligned}
			\sup_{M, N \in \NN} L^{\kappa/4  N}	\Big( \max_{1\leqslant
				n\leqslant N+M} L^{- \kappa  n}  \sup_{x\in \Lambda^{n}}
	\big| 
	\big(z_{-n+1}^{(N+M)}- z_{-n+1}^{(N)}\big)
	(x)\big|\Big)< \infty\,,
		\end{aligned}
	\end{equation}
with $ \kappa= 3 \kappa_{\rms}$.
	As we have done so far, we also follow here the proof of the analogue in
	Section~\ref{sec_rig_RG},
	namely the proof of Lemma~\ref{lem_aid_besov}, to control the differences of the
	two fluctuation fields. We avoid repetition and keep the argument short.

\begin{proof}[Proof of Lemma~\ref{lem_supp_conv_v}]
		First, by means of \eqref{eq_supp2_z} and $ \| P_{1}f \|_{\infty} \leqslant 2 \| f \|_{\infty}$, we have 
	\begin{equs}
		\begin{aligned}
			&\| z_{-n+1}^{(N+M)} - z_{-n+1}^{(N)}\|_{\infty}\\
			&\leqslant 
			2 \big(\lambda_{-n} \| \xi_{-n} \|_{\infty}+  \lambda_{-n}^{2} \| \tOne^{(N)}_{-n}\|_{\infty} +
				\lambda_{-n}^{3} \| \Psi_{-n}^{(N)} \|_{\infty} +
			\mu_{-n}\big) \big(\| z_{-n+1}^{(N+M)} -
			z_{-n+1}^{(N)}\|_{\infty}+\| w_{-n+1}^{(N+M)} - w_{-n+1}^{(N)}\|_{\infty} \big)\\
			& \quad + 	\mu_{-n} \| z_{-n+1}^{(N+M)} -z_{-n+1}^{(N)}\|_{\infty}\\
			& \quad + 2 \big(  \lambda_{-n}^{2} \|\tOne^{(N+M)}_{-n}- \tOne^{(N)}_{-n}\|_{\infty} +
				\lambda_{-n}^{3} \| \Psi_{-n}^{(N+M)}-\Psi_{-n}^{(N)}
			\|_{\infty}\big) \big(\| z_{-n +1}^{(N+M)} \|_{\infty}+ \|
		w_{-n +1}^{(N+M)}\|_{\infty}\big)\,.
		\end{aligned}
	\end{equs}
	Using Lemmas~\ref{lem_stoch},~\ref{lem_fp} and~\ref{lem_aid_besov}, and
	rearranging the terms appropriately,
	yields  
	\begin{equation*}
		\begin{aligned}
			\| z_{-n+1}^{(N+M)} - z_{-n+1}^{(N)}\|_{\infty}
			&\leqslant 
2 \Big( 4L^{ (-  \alpha + \kappa_{\rms}) n}\| w_{-n+1}^{(N+M)} -
				w_{-n+1}^{(N)}\|_{\infty}\\
			& \quad\quad + 16 L^{(-\alpha+ 3 \kappa_{\rms}) n } g^{2} \big(\|\tOne^{(N+M)}_{-n}- \tOne^{(N)}_{-n}\|_{\infty}
+\lambda_{-n}	\| \Psi_{-n}^{(N+M)}-\Psi_{-n}^{(N)}
			\|_{\infty}
			\big)\Big)\,.
       		\end{aligned}
	\end{equation*}
	Notice that the second line can be bounded by an arbitrary small constant
	for $N$ large enough, see Lemmas~\ref{lem_stoch_conv} and~\ref{lem_remainder_conv}.
	Thus, we have an estimate analogous to~\eqref{eq_bound_zWithQ}.

	The remainder of the proof follows by the same induction argument as for Lemma~\ref{lem_aid_besov}.
\end{proof}

Finally, we can summarise our findings in the following proof:

\begin{proof}[Proof of Proposition~\ref{prop_main}]
	Restricting ourselves to the event $ \Omega_{g}$, 
	we proved that $ v^{(N)}$ converges to $ v^{\star} = (v^{\star}_{-n})_{n \in \NN}
	$, see Lemma~\ref{lem_supp_conv_v}. Moreover,  we established convergence of the
	force coefficients 
	\begin{equation*}
		\begin{aligned}
			\lim_{ N \to \infty} \tOne_{\!\! -n}^{(N)}=
			\tOne^{\star}_{\!\!-n}\,, \quad \text{and} \quad 
			\lim_{N \to \infty} \Psi_{-n}^{(N)}= \Psi^{\star}_{-n}\,,
		\end{aligned}
	\end{equation*}
	in Lemmas~\ref{lem_stoch_conv} and~\ref{lem_remainder_conv}.
	It is only left to notice that $ v_{-n}^{\star}$  solves indeed the limiting
	equation 
	 \begin{equation*}
        \begin{aligned}
            (- \Delta_{H}) v_{-n}^{\star} =
            \big(\lambda_{-n} \xi_{-n}
            + \lambda_{-n}^{2} \tOne_{\! \! -n}^{\, \star}
            + \lambda_{-n}^{3} \Psi^{\star}_{-n} + \mu_{-n}\big)v^{\star}_{-n}  +
	    \gamma_{-n}\,,
        \end{aligned}
    \end{equation*}
	since the equation can be written in terms of finite linear combinations of all
	quantities involved. 
    Notice that the constants $ \lambda_{-n}, \mu_{-n}$, and $ \gamma_{-n}$ were
    independent of $N$.
    Lastly, we note that $ v^{(N)}_{-n} = L^{- \alpha (m-n)} S_{m-n} Q_{m-n}^{(N)} v^{(N)}_{-m}$ for every
    $n <m \leqslant N$, which remains true when taking limits.
\end{proof}

\appendix
 
\section{Moment calculations}

In this appendix, we collect several moment calculations which are useful
throughout the paper.

\begin{lemma}\label{lem_expect_db}
    Let $N \in \NN$, then
    for every $n <N $
    \begin{equation*}
        \begin{aligned}
            \EE\big[
                Q_{ N-n} ( \xi_{-N} \Gamma_{N-n}
                \xi_{-N}) (x)
                \big] =
            ( 1- L^{-d})
            \sum_{k =1}^{N-n} L^{(2-d)(k-1)}
            \,, \qquad \forall x \in \Lambda^{N}_{N-n}\,,
        \end{aligned}
    \end{equation*}
    where $ \Gamma_{N-n}$ was defined in \eqref{eq_def_Gamma}.
    In particular, $ \EE\big[
            Q_{ 1} ( \xi_{-n} P_{1}
            \xi_{-n})
            \big] = ( 1- L^{-d})$.
\end{lemma}
\begin{proof}
    First, we notice that due to spatial homogeneity (in law), for every $ x \in
        \Lambda^{N}_{N-n}$
    \begin{equation*}
        \begin{aligned}
            \EE\big[ Q_{ N-n} ( \xi_{-N} \Gamma_{N-n}
                \xi_{-N}) (x) \big]
            =
            \EE\big[  \xi_{-N} (0) \Gamma_{N-n}
                \xi_{-N} (0)\big]
            =
            \sum_{k =1}^{N-n} L^{2(k-1)} \EE[ \xi_{-N}(0) P_{k} \xi_{-N}(0)]\,.
        \end{aligned}
    \end{equation*}
    Next, we observe that
    \begin{equation}\label{eq_supp1_dbmean}
        \begin{aligned}
            \EE[ \xi_{-N}(0) P_{k} \xi_{-N} (0)]
            =
            \EE[ \xi_{-N}(0) (Q_{k-1}-Q_{k}) \xi_{-N}(0)]
            =
            L^{-d(k-1)} (1- L^{- d})\,,
        \end{aligned}
    \end{equation}
    where we used that $ \xi_{-N}$ is a field of independent unit variance random
    variables, which implies
    \begin{equation}\label{eq_supp2_dbmean}
        \begin{aligned}
            \EE[ \xi_{-N}(0) Q_{ k} \xi_{-N}(0)]
            =
            L^{- d k}
            \sum_{x \in \bigsquare_{k}(0) }
            \EE [ \xi_{-N}(0) \xi_{-N}(x)]
            = L^{-dk}\,.
        \end{aligned}
    \end{equation}
    Recall that $ \bigsquare_{k}(0)$ denotes the block of side--length
    $L^{k}$ centred at $0$.
    Combining \eqref{eq_supp1_dbmean} and \eqref{eq_supp2_dbmean} yields
    \begin{equation*}
        \begin{aligned}
            \EE\big[ Q_{ N-n} ( \xi_{-N} \Gamma_{N-n}
                \xi_{-N}) (x) \big]
            =
            ( 1- L^{-d})
            \sum_{k =1}^{N-n} L^{(2-d)(k-1)}
            \,,
        \end{aligned}
    \end{equation*}
    which concludes the first statement.
    Lastly, we notice that \eqref{eq_supp1_dbmean} implies the second statement of
    the lemma by choosing $k=1$ and $N=n$.
\end{proof}

\begin{lemma}\label{lem_cov_P1xi}
    For every $n \in \NN$ and $ x \in \Lambda^{n}$, we have
    \begin{equation*}
        \begin{aligned}
            \EE \big[
                P_{1} \xi_{-n} (x) P_{1} \xi_{-n} (0)
                \big]
            =
            \begin{cases}
                \delta_{x,0} - L^{-d}\,, & \quad \text{if } x \in \bigsquare_{1} (0)\,, \\
                0\,,                     & \quad \text{otherwise}\,.
            \end{cases}
        \end{aligned}
    \end{equation*}
\end{lemma}

\begin{proof}
    Clearly, if $ x $ and $0$ do not lie in the same box of side--length $L$, i.e.
    $ x \notin \bigsquare_{1} (0)$, then the two random variables are independent.
    On the other hand, if $ x \in \bigsquare_{1}(0)$, then $ \bigsquare_{1} (x) =
        \bigsquare_{1} (0)$ and
    \begin{equation*}
        \begin{aligned}
             & \EE \big[
                P_{1} \xi_{-n} (x) P_{1} \xi_{-n} (0)
            \big]                           \\
             & =
            \EE \bigg[
                \Big( L^{-d} \sum_{ y \in \bigsquare_{1}(0) \cap \Lambda^{n}} \xi_{-n} (y) -  \xi_{-n} (x) \Big)
                \Big(L^{-d}
                \sum_{ z \in \bigsquare_{1}(0) \cap \Lambda^{n}} \xi_{-n} (z)  - \xi_{-n}(0)\Big)
            \bigg]                          \\
             & =
            L^{-2d}
            \sum_{ y ,z \in \bigsquare_{1}(0) \cap \Lambda^{n}}
            \EE [ \xi_{-n} (y)\xi_{-n} (z)]
            + \EE [ \xi_{-n} (x) \xi_{-n} (0)]
            - 2 L^{-d}
            \sum_{ y \in \bigsquare_{1}(0) \cap \Lambda^{n}}
            \EE [ \xi_{-n} (y)\xi_{-n} (0)] \\
             & =
            L^{-d}+
            \delta_{x,0}
            - 2 L^{-d} = \delta_{x, 0}- L^{-d}\,,
        \end{aligned}
    \end{equation*}
    where we simply used that $ ( \xi_{-n} (x) )_{x \in \Lambda^{n}}$ are
    independent standard normal random variables.
\end{proof}

\begin{corollary}\label{cor_cov_P1db}
    For every $ n \in \NN$ and $ x \in \Lambda^{n}$,
    \begin{equation}\label{eq_supp2_db_cv}
        \begin{aligned}
            \EE \big[
                (P_{1}
                \xi_{-n})^{\diamond 2}(x)
                (P_{1}
                \xi_{-n})^{\diamond 2}(0)\big]
            =\begin{cases}
                 2 ( \delta_{x,0}- L^{-d} )^{2}\,, & \quad \text{if }   x\in \bigsquare_{1} (0)\,, \\
                 0\,,                              & \quad \text{otherwise}\,.
             \end{cases}
        \end{aligned}
    \end{equation}
\end{corollary}

\begin{proof}
    The statement is an immediate consequence of  the definition of
    the Wick ordering:
    \begin{equation*}
        \begin{aligned}
            \EE \big[
                (P_{1}
                \xi_{-n})^{\diamond 2}(x)
                (P_{1}
                \xi_{-n})^{\diamond 2}(0)\big]
            =
            2
            \big(
            \EE \big[
                (P_{1} \xi_{-n})(x)
                (P_{1}\xi_{-n})(0)\big]
            \big)^{2}\,,
        \end{aligned}
    \end{equation*}
and    Lemma~\ref{lem_cov_P1xi}.
\end{proof}


\end{document}